\documentclass[11pt,english]{article}

\RequirePackage{a4wide}

\RequirePackage{amsmath,amsfonts,amssymb}
\RequirePackage{amsthm}
\RequirePackage{mathtools}
\RequirePackage{mathrsfs}
\RequirePackage{bbm}
\usepackage[shortlabels]{enumitem}
\usepackage{soul}
\usepackage{cancel}
\setenumerate{label=(\roman*),itemsep=3pt,topsep=3pt}

\usepackage[pdftex,dvipsnames]{xcolor}  
\RequirePackage[colorlinks,linkcolor=blue,citecolor=blue,urlcolor=blue]{hyperref}
\RequirePackage{cleveref} 
\RequirePackage{graphicx}

\RequirePackage[numbers]{natbib} 

\RequirePackage{algpseudocode}
\RequirePackage{algorithm}
\RequirePackage{nicematrix}

\RequirePackage{caption}
\RequirePackage{subcaption}
\RequirePackage{tikz-cd}
\RequirePackage{adjustbox}

\RequirePackage{comment}
\RequirePackage{xargs}
\newcommandx{\edit}[2][1=]
    {\todo[linecolor=Blue,backgroundcolor=Blue!10,bordercolor=Blue,#1]{#2}}

\definecolor{hblue}{RGB}{0, 102, 204}
\definecolor{fblue}{RGB}{0, 51, 102}
\definecolor{hcolor}{RGB}{207, 68, 142}
\definecolor{qcolor}{RGB}{255, 138, 220} 

\usepackage{mdframed}

\numberwithin{equation}{section}
\theoremstyle{plain}
\newtheorem{theorem}{Theorem}[section]
\newtheorem{lemma}[theorem]{Lemma}

\newtheorem{assumption}[theorem]{Assumption}

\theoremstyle{remark}

\newtheorem{example}[theorem]{Example}
\newtheorem{remark}[theorem]{Remark}
\renewcommand{\theta}{\vartheta}
\newcommand*\diff{\mathop{}\!\mathrm{d}}
\DeclareMathOperator{\PP}{{\mathbb P}}

\DeclareMathOperator{\R}{{\mathbb R}}
\DeclareMathOperator{\N}{{\mathbb N}}
\DeclarePairedDelimiter\ceil{\lceil}{\rceil}

\newcommand{\sca}[2]{\langle#1,#2\rangle}
\newcommand{\norm}[1]{\lVert#1\rVert}
\numberwithin{theorem}{section}
\numberwithin{figure}{section}
\title{Parameter estimation in hyperbolic linear SPDEs from multiple measurements}
\author{Anton Tiepner, Eric Ziebell}
\date{}

\begin{document}

\maketitle
\begin{abstract}
    \noindent The coefficients of elastic and dissipative operators in a linear hyperbolic SPDE are jointly estimated using multiple spatially localised measurements. As the resolution level of the observations tends to zero, we establish the asymptotic normality of an augmented maximum likelihood estimator. 
    The rate of convergence for the dissipative coefficients matches rates in related parabolic problems, whereas the rate for the elastic parameters also depends on the magnitude of the damping. The analysis of the observed Fisher information matrix relies upon the asymptotic behaviour of rescaled $M, N$-functions generalising the operator cosine and sine families appearing in the undamped wave equation. In contrast to the energetically stable undamped wave equation, the $M, N$-functions emerging within the covariance structure of the local measurements
    have additional smoothing properties similar to the heat kernel, and their asymptotic behaviour is analysed using functional calculus. 
    \medskip


\smallskip

\end{abstract}

\section{Introduction}
We study parameter estimation for a general second-order stochastic Cauchy problem 
\begin{equation}
    \label{eq: genCauchyIntro}
    \ddot{u}(t)=A_\theta u(t)+B_\eta \dot{u}(t)+\dot{W}(t),\quad 0<t\leq T,
\end{equation}
driven by space-time white noise $\dot{W}$ on an open, bounded spatial domain $\Lambda\subset\R^d$. The differential operators $A_\theta$ and $B_\eta$ defined through
\begin{align*}
    A_\theta&=\sum_{i=1}^p\theta_i(-\Delta)^{\alpha_i},\quad \alpha_1>\dots>\alpha_p\geq0,\\\
    B_\eta&=\sum_{j=1}^q\eta_j(-\Delta)^{\beta_j},\quad\beta_1>\dots>\beta_q\geq0,
\end{align*}
are parameterised by unknown constants $\theta\in\R^p,$ $\eta\in\R^q$. In general, such equations model elastic systems, and we refer to $A_\theta$ as the elastic operator while $B_\eta$ is called the dissipation (or damping) operator.\par In the absence of any damping $(B_\eta=0)$ and noise, a prototypical example of \eqref{eq: genCauchyIntro} is the isotropic plate equation (without any in-plane forces, thermal loads or elastic foundation)
\begin{equation}
    \label{eq: detplate}
    \rho h\ddot{u}(t)=-D\Delta^2u(t),
\end{equation}
modelling the bending of elastic plates over time. The parameters governing \eqref{eq: detplate} are the material density $\rho$ and the bending stiffness $D=\frac{h^3E}{12(1-\nu^2)}$. The bending stiffness $D$ depends on the plate thickness $h$, the material-specific Poisson-ratio $\nu$ and Young's modulus $E$. Numerous extensions and applications of such equations can, for instance, be found in \cite{reddy_plate_2006,leissa_plate_1969}. To account for the system's energy loss, damping is added to the equation, where a higher differential order of $B_{\eta}$ describes stronger damping. In fact, a parabolic behaviour of \eqref{eq: genCauchyIntro} is obtained for $\beta_1>0$ due to the smoothing effects within the related $C_0$-semigroup \cite{russell_damping_1982, Favini_conditions_1991}. In closely related situations, i.e. when both $A_\theta$ and $B_\eta$ are negative operators, the $C_0$-semigroup has been shown to become analytic if and only if $2\beta_1\geq\alpha_1$, where a borderline case occurs under equality, cf. \cite{chen_damping_1989}.\par While parameter estimation for SPDEs is well-studied in second-order parabolic equations, e.g. \cite{altmeyer_anisotrop2021, lototsky_statistical_2009, hildebrandt_parameter_2019, tonaki2022parameter, chong_high-frequency_2020,gaudlitz_estimation_2023,pasemann_drift_2020} and the references therein, the literature on higher-order hyperbolic equations is limited. We refer to \cite{ziebell_wave23} and the references mentioned there for studies of the (non-)parametric wave equation. In \cite{aihara_parameter_1991,aihara_identification_1994}, the authors considered a \textit{weakly damped} system, i.e. $\beta_1=0$, and developed a first approach for identifying coefficients of the elastic operator in a Kalman filtering problem based on the methods of sieves. \cite{maslowski_ergodicity_2007} studied equations driven by a fractional cylindrical Brownian motion and derived a consistent estimator of a scalar drift coefficient using the ergodicity of the underlying system. Based on spectral measurements $(\sca{u(t)}{e_j})_{0\leq t\leq T}$, $j=1,\dots, N$, where $(e_j)_{j\geq1}$ forms an orthonormal basis of $L^2(\Lambda)$ composed of eigenvectors for $A_\theta$ and $B_\eta$, \cite{lototsky_multichannel_2010} constructed maximum-likelihood estimators and established the asymptotic normality for diagonalisable hyperbolic equations given that the number $N$ of observed Fourier-modes tends to infinity.

In contrast, our estimator is based on continuous observations of local measurement processes $$u_{\delta,k}=(\sca{u(t)}{K_{\delta,x_k}})_{0\leq t\leq T},\quad u_{\delta,k}^\gamma=(\sca{u(t)}{(-\Delta)^{\gamma}K_{\delta,x_k}})_{0\leq t\leq T},$$ for locations $x_1,\dots, x_N\in\Lambda$ and $\gamma\in\{\alpha_i,\beta_j|1\leq i\leq p;1\leq j\leq q\}$. The \textit{point spread functions} \cite{aspelmeier_modern_2015,backer_extending_2014} $K_{\delta,x_k}$ are compactly supported functions taking non-zero values in an area centred around $x_k$ with radius $\delta$. Local measurements emerge naturally as they describe the physical limitation of measuring $u(t,x_k)$, which, in general, is only possible up to a convolution with a point spread function. Local observations were introduced to the field of statistics for SPDEs in \cite{altmeyer_nonparametric_2020}, where the authors investigated a stochastic heat equation with a spatially varying diffusivity. It was shown that the diffusivity at location $x_k\in\Lambda$ can be estimated based on a single local measurement process at $x_k$ as the resolution level $\delta$ tends to zero. The local observation scheme turned out to be robust under semilinearities \cite{altmeyer_parameterSemi_2020, altmeyer_parameter_2020}, multiplicative noise \cite{janak_2023_multiplicative}, discontinuities \cite{reiß_2023_change,tiepner_change_24} or lower-order perturbation terms \cite{altmeyer_anisotrop2021,strauch_velocity_2023}. In contrast to the estimation of the diffusivity, the identifiability of transport or reaction coefficients necessarily requires an increasing amount $N\rightarrow\infty$ of measurements. In the recent contribution \cite{ziebell_wave23}, the local measurement approach was extended to hyperbolic problems and the non-parametric wave speed in the undamped stochastic wave equation was estimated by relating the observed Fisher information to the energetic behaviour of an associated deterministic wave equation.

Based on the local measurement approach, we construct the \textit{augmented} maximum likelihood estimator (MLE) $(\hat{\theta}_\delta,\hat{\eta}_\delta)^\top\in\R^{p+q}$ and prove the asymptotic normality of 
\begin{equation*}
\label{eq: CLT intro}
    \begin{pmatrix}
        N^{1/2}\delta^{-2\alpha_i+\alpha_1+\beta_1}(\hat{\theta}_{\delta,i}-\theta_i)\\
        N^{1/2}\delta^{-2\beta_j+\beta_1}(\hat{\eta}_{\delta,j}-\eta_j)
    \end{pmatrix}_{i\leq p,j\leq q},\quad\delta\rightarrow0,
\end{equation*}
with $N=N(\delta)$ measurements. The consistent estimation for $\theta_i$ holds if $N^{1/2}\delta^{-2\alpha_i+\alpha_1+\beta_1}\rightarrow\infty$, whereas $\eta_j$ can be estimated in the asymptotic regime $N^{1/2}\delta^{-2\beta_j+\beta_1}\rightarrow\infty$. In particular, estimating elastic coefficients is more difficult under higher dissipation, while damping coefficients are unaffected by the order of the elastic operator $A_\theta$ and their convergence rates reflect the rates obtained in advection-diffusion equations, cf. \cite{altmeyer_anisotrop2021}. For the maximal number of non-overlapping observations $N\asymp\delta^{-d}$, our convergence rates match the rates obtained in the spectral approach up to specific boundary cases. In the weakly damped case, i.e. $\beta_1=0,$ we confirm the results in \cite{lototsky_multichannel_2010}. That is, the dependence of the time horizon $T$ of the asymptotic variance resembles the explosive, stable and ergodic cases of the maximum likelihood drift estimator for an Ornstein--Uhlenbeck process, cf. \cite[Proposition 3.46]{kutoyants_statistical_2013}. In the structural damped case $(\beta_1>0)$, the asymptotic variance is of order $T^{-1}$ instead.

The central limited theorem follows conceptually through a martingale CLT and Slutsky's lemma from the stochastic convergence of the rescaled observed Fisher information matrix. Verification of this convergence, however, is nontrivial and requires new approaches due to the structure of the semigroup associated with \eqref{eq: genCauchyIntro}. It is characterised by $M, N$-functions, cf. \cite[Chapter 1.7]{melnikova_abstract_2001}, which are generalisations of operator cosine and sin functions appearing in the solution to the stochastic wave equation. It is a-priori unclear whether the local measurement scheme is applicable in this setting since the involved $M, N$-functions inherit interacting properties from both the heat and the energetically stable wave equation. Smoothing and damping effects from the heat equation are introduced to the $M, N$-functions through the semigroup $(e^{tB_\eta})_{t\geq0}$, while an oscillating behaviour is mirrored by terms of the form $(\cos(t (-A_\theta-B_\eta^2/4)^{1/2}))_{t\geq0}$ and $(\sin(t (-A_\theta-B_\eta^2/4)^{1/2})(-A_\theta-B_\eta^2/4)^{-1/2})_{t\geq0}$, respectively, which are the operator cosine and sine functions generated by $-A_\theta-B_\eta^2/4$. Consequently, both the elastic operator and the dissipative operator contribute to the oscillating component's behaviour. The underlying solution operator combines phenomena from both heat and wave equations, which, surprisingly, results in a convergence proof that cannot directly combine the steps required in the asymptotic analysis for either of them. Instead, we exploit functional calculus for operators, which relies on a detailed analysis of the underlying $M, N$-functions and their rescaling when applied to localising kernel functions $K_{\delta,x}$. To the best of our knowledge, we are the first to utilise the theory of $M, N$-functions in a statistical field. Thus, we believe that our findings strengthen the existing toolbox of proof techniques and provide beneficial insights for future work in statistics of higher-order SPDEs.

We begin this paper by specifying the model and discussing properties of the local measurements in \Cref{sec: setup}. The augmented MLE is constructed and analysed in \Cref{sec: estimator}, and the CLT is established. The section additionally contains various remarks and examples, complemented by a numerical study underpinning and illustrating the main result. All proofs are deferred to \Cref{sec: proofs}.
\section{Setup}
\label{sec: setup}
\subsection{Notation}
Throughout this paper, we fix a filtered probability space $(\Omega, \mathcal{F},(\mathcal{F}_t)_{0\leq t\leq T},\mathbb{P})$ with a fixed time horizon $T<\infty$. 
We write $a\lesssim b$ if $a\leq Mb$ holds for a universal constant $M$, independent of the resolution level $\delta>0$ and the number of spatial points $N$. Unless stated otherwise, all limits are to be understood as the spatial resolution level tending to zero, i.e. for $\delta\rightarrow 0$. 
For an open set $\Lambda\subset\R^d$, $L^2(\Lambda)$ is the usual $L^2$-space with the inner product $\sca{\cdot}{\cdot}\coloneqq\sca{\cdot}{\cdot}_{L^2(\Lambda)}$. 
The Euclidean inner product and distance of two vectors $a, b\in\R^p$ are denoted by $a^\top b$ and $|a-b|$, respectively. We abbreviate the Laplace operator with Dirichlet boundary conditions on the bounded spatial domain $\Lambda$ by $\Delta$ and on the unbounded spatial domain $\R^d$ by $\Delta_0$.
 Let $H^k(\Lambda)$ denote the usual Sobolev spaces, and denote by $H_0^1(\Lambda)$ the completion of $C_c^{\infty}(\Lambda)$, the space of smooth compactly supported functions, relative to the $H^1(\Lambda)$ norm. As in \cite{kovacs_finite_2010}, let $\dot H^{2s}(\Lambda)\coloneqq\mathcal{D}((-\Delta)^s)$ for $s>0$ be the domain of the fractional Laplace operator on $L^2(\Lambda)$ with Dirichlet boundary conditions. The order of a differential operator $D$ is denoted by $\operatorname{ord}(D).$
\subsection{The model}
Consider the second-order stochastic Cauchy problem
\begin{equation}
    \label{eq: genspde}
    \begin{cases}
        \diff v(t)=\left(A_\theta u(t)+B_\eta  v(t)\right)\diff t+\diff W(t),\quad 0<t\leq T,\\
        \diff u(t)=v(t)\diff t,\\
        u(0)=u_0\in L^2(\Lambda),\\
        v(0)=v_0\in L^2(\Lambda),\\
        u(t,x)=v(t,x)=0,\quad 0\leq t\leq T,\quad x\in\Lambda|_{\partial\Lambda}
    \end{cases}
\end{equation}
on an open, bounded domain $\Lambda\subset\R^d$ having $C^2$-boundary $\partial\Lambda$. We assume Dirichlet boundary conditions and a driving space-time white noise $\diff W$ in \eqref{eq: genspde}. The elasticity and damping operators $A_\theta$ and $B_\eta$ are parameterised by $\theta\in\R^p$ and $\eta\in\R^q$ and given by
\begin{equation}
    \label{eq: opBA}
    \begin{aligned}
        A_\theta&=\sum_{i=1}^p\theta_i(-\Delta)^{\alpha_i}, \quad  D(A_{\vartheta})= D((-\Delta)^{\alpha_1}) = \dot{H}^{2\alpha_1}(\Lambda) ,\\
        B_\eta&=\sum_{j=1}^q\eta_j(-\Delta)^{\beta_j}, \quad  D(B_{\eta})=D((-\Delta)^{\beta_1})= \dot{H}^{2\beta_1}(\Lambda),
    \end{aligned}
    \end{equation}
with $\alpha_1>0$ and $0\leq \alpha_i,\beta_j<\infty$ satisfying $\alpha_1>\alpha_2>\dots >\alpha_p$ and $\beta_1>\beta_2>\dots >\beta_q$.
\begin{example}
\label{ex: heatmap}
\begin{enumerate}
\item[]
\item [(a)]  Weakly damped wave equation ($\beta_1=0$): $A_\theta=-\theta_1\Delta$, $B_\eta=\eta_1.$
\item [(b)] Clamped plate equation: 
\begin{itemize}
    \item [1)] Weakly damped $(\beta_1=0):$ $A_\theta=\theta_1\Delta^2$, $B_\eta=\eta_1.$
    \item [2)] Structurally damped ($0<\beta_1<\alpha_1$): $A_\theta=\theta_1\Delta^2$, $B_\eta=-\eta_1\Delta.$
    \item [3)] Strongly damped ($\beta_1=\alpha_1$): $A_\theta=\theta_1\Delta^2$, $B_\eta=\eta_1\Delta^2.$
\end{itemize}
\Cref{fig: 1} displays a heatmap illustrating both the weakly and structurally damped plate equation in one spatial dimension. The solution of the SPDEs were approximated on a fine time-space grid using the finite difference scheme associated with the semi-implicit Euler-Maruyama method, see \cite[Chapter 10]{lord_introduction_2014}. Additional smoothing properties in the structurally damped case due to the dissipative operator $B_\eta$ result in an accelerated energetic decay in comparison to the weakly damped case.
\end{enumerate}
\end{example}
\begin{figure}[t]
	\centering
		\centering\includegraphics[width=1.\linewidth]{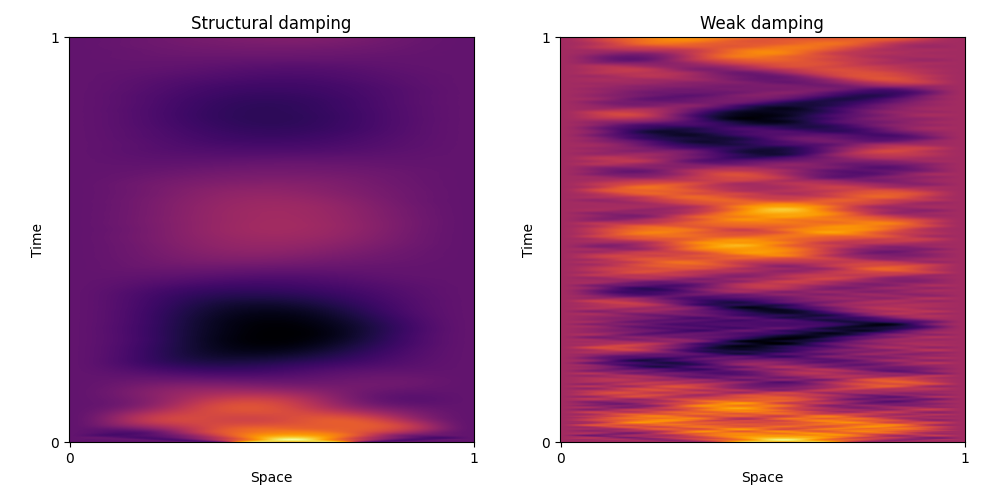}
	\caption{Realisation of the solution $u(t,x)$ to the clamped plate equation on $(0,1)$; (left) $\ddot u(t)=-0.3\Delta^2u(t)+0.3\Delta\dot u(t)+\dot W(t)$; (right) $\ddot u(t)=-0.3\Delta^2u(t)-0.3\dot u(t)+\dot W(t)$.}
 \label{fig: 1}
\end{figure}
\newpage \noindent Throughout the rest of the paper, we impose the following assumptions on the parameters.
\begin{assumption}[{Assumption on the parameters}]
\label{ass: parameterass}
    \begin{itemize}
        \item[]
        \item [(i)] $\theta_1<0;$
        \item [(ii)] If $\beta_1>0,$ then $\eta_1<0;$
        \item [(iii)] $\alpha_1\geq 2\beta_1$ and if $\alpha_1=2\beta_1$ then $\theta_1+\eta_1^2/4<0.$
    \end{itemize} 
\end{assumption}
\Cref{ass: parameterass} does not guarantee that either $A_\theta$ or $B_\eta$ are, in general, negative operators. Instead, it implies that at least all but finitely many of their eigenvalues are negative. Moreover, conditions (i) and (iii) are necessary to ensure that the difference $L_{\theta,\eta}\coloneqq -A_\theta-B_\eta^2/4$ is a positive operator, which itself is not required for the proofs and is just assumed for technical reasons in \Cref{sec: proofs} as all arguments also carry over to the non-positive case, resulting in a complex-valued operator, cf. \Cref{rmk: complexOp}.
\subsection{Local measurements}
For $\delta>0$, $y\in\Lambda$ and $z\in L^2(\R^d)$ we define the rescaling 
\begin{equation}
\label{eq:shift_scaling_operator}
    \begin{aligned}
        \Lambda_{\delta,y}&=\{\delta^{-1}(x-y):x\in\Lambda\},\\
        z_{\delta,y}(x)&=\delta^{-d/2}z(\delta^{-1}(x-y)),\quad x\in\R^d.
    \end{aligned}
\end{equation}
Fix a function $K\in H^{\lceil2\alpha_1\rceil}(\R^d)$ with compact support. By a slight abuse of notation, we define local measurements at the location $x\in\Lambda$ with resolution level $\delta$ as the continuously observed processes $u_{\delta,x}, u^{\Delta_i}_{\delta,x}, v_{\delta,x}, v_{\delta,x}^{\Delta_j}$ where for $i=1,\dots,p$ and $j=1,\dots,q$:
\begin{align*}
    u_{\delta,x} &=(\sca{u(t)}{K_{\delta,x}})_{0\leq t\leq T}, \quad 
    u^{\Delta_i}_{\delta,x}=(\sca{u(t)}{(-\Delta)^{\alpha_i} K_{\delta,x}})_{0\leq t\leq T},\\
    v_{\delta,x}&=(\sca{v(t)}{K_{\delta,x}})_{0\leq t\leq T}, \quad 
    v_{\delta,x}^{\Delta_j}=(\sca{v(t)}{(-\Delta)^{\beta_j}K_{\delta,x}})_{0\leq t\leq T}.
\end{align*}
Analogously to \cite{ziebell_wave23}, these local measurements satisfy the following Itô-dynamics
\begin{equation}
    \label{eq: dynamics}
    \diff u_{\delta,x}(t)=v_{\delta,x}(t)\diff t,\quad \diff v_{\delta,x}(t)=\left(\sum_{i=1}^p\theta_iu^{\Delta_i}_{\delta,x}(t)+\sum_{j=1}^q\eta_jv^{\Delta_j}_{\delta,x}(t)\right)\diff t+ \norm{K}_{L^2(\R^d)}\diff W_x(t),
\end{equation}
with scalar Brownian motions $(W_x(t))_{0 \leq t \leq T}=(\norm{K}^{-1}_{L^2(\R^d)}\sca{W(t)}{K_{\delta,x}})_{0\leq t\leq T}$, which become mutually independent provided that
\begin{equation*}
\sca{K_{\delta,x}}{K_{\delta,x'}}=\norm{K}^2_{L^2(\R^d)}\delta_{x,x'}=0,\quad x,x'\in\R^d,
\end{equation*}
where the Kronecker-delta $\delta_{x,x'}$ evaluates to zero for $x \neq x'$.

For locations $x_1,\dots, x_N\in\Lambda$, we further define the observation vector process $Y_{\delta}\in L^2([0,T];\R^{(p+q)\times N})$ through
\begin{equation}
\label{eq: observation vector}
    Y_{\delta,k} = \begin{pmatrix}
        u_{\delta,x_k}^{\Delta_1}&
        \hdots&
        u_{\delta,x_k}^{\Delta_p}&
        v_{\delta,x_k}^{\Delta_1}&
        \hdots&
        v_{\delta,x_k}^{\Delta_q}
    \end{pmatrix}^{\top}\in\R^{p+q}, \quad k=1, \dots, N.
\end{equation}
\begin{remark}[{Accessibility of the measurements}]
The measurements $(u_{\delta,x}^{\Delta_i}$, $i=1, \dots, p)$, can be approximated by observing $u_{\delta,y}$, $y\in\Lambda,$ on a fine spatial grid in $\Lambda.$  Moreover, all the measurements wrt. $v$, i.e. $v_{\delta,x}$ and $v^{\Delta_j}_{\delta,x}$, $j\leq q$, can be obtained by differentiating $u_{\delta,x}$ and $u^{\Delta_j}_{\delta,x}\coloneqq\sca{u(\cdot)}{(-\Delta)^{\beta_j}K_{\delta,x}}$ in time. 
\end{remark}
A proper understanding of the covariance structure of the local measurements \eqref{eq: observation vector} is a key factor in the asymptotic analysis of the constructed estimator in \Cref{sec: estimator}.
Denoting $L_{\theta,\eta}=-A_\theta-B_\eta^2/4$, the variance of a local measurement satisfies for instance
\begin{align}
\label{eq: var_tease}
    \operatorname{Var}(u^{\Delta_i}_{\delta,x}(t))=\int_0^t\norm{e^{sB_\eta}\sin(s L_{\theta,\eta}^{1/2})L_{\theta,\eta}^{-1/2}(-\Delta)^{\alpha_i}K_{\delta,x}}^2\diff s,\quad 1\leq i\leq p,\quad 0<t\leq T.
\end{align}
Intriguingly, phenomena from both heat and wave equations are infused in the representation of the variance \eqref{eq: var_tease}. Indeed, in the heat equation (cf. \cite{altmeyer_nonparametric_2020}), the variance structure is determined by the semigroup $(e^{t\Delta})_{t\geq0}$ generated by the Laplacian $\Delta$. It has a smoothing behaviour and induces the system's energetic decay. In the wave equation (cf. \cite{ziebell_wave23}), the corresponding variance term relies on the oscillating operator sine function $(\sin(t(-\Delta)^{1/2})(-\Delta)^{-1/2})_{t\geq0}$ generated by $-\Delta$. In \eqref{eq: var_tease}, both effects are present through the differential operators $B_\eta$ and $L_{\theta,\eta}$ which leads to the study of $M,N$-functions generalising operator cosine and sine functions. When applied to localising functions $K_{\delta,x},$ a rescaling occurs, which takes both effects into account. Even though the limits coincide, asymptotic convergence of \eqref{eq: var_tease} for $\delta\rightarrow0$ follows not through the integrated convergence of the heat kernel or the pointwise convergence (asymptotic equipartition of energy) known from the wave equation. Instead, we have to take a more subtle approach utilising functional calculus for operators as presented in \Cref{sec: proofs}.

\section{The estimator}
\label{sec: estimator}
Motivated by a general Girsanov theorem, as described in detail in \cite{altmeyer_nonparametric_2020,altmeyer_anisotrop2021}, the augmented MLE $(\hat{\theta}_{\delta},\hat{\eta}_{\delta})^\top\in\R^{p+q}$ is given by
\begin{equation}
    \label{eq: jointest}
    \begin{pmatrix}
        \hat{\theta}_\delta\\
        \hat{\eta}_\delta
    \end{pmatrix}=\mathcal{I}_{\delta}^{-1}\sum_{k=1}^N\int_0^TY_{\delta,k}(t)\diff v_{\delta,x_k}(t)
\end{equation}
with the observed Fisher information matrix 
\begin{equation}
    \label{eq: FisherJoint}
    \mathcal{I}_{\delta}=\sum_{k=1}^N\int_0^TY_{\delta,k}(t)Y_{\delta,k}(t)^\top\diff t.
\end{equation}
Clearly, the matrix $\mathcal{I}_\delta$ is symmetric and positive semidefinite.
By plugging the Itô-dynamics \eqref{eq: dynamics} into the definition of the estimator \eqref{eq: jointest}, we obtain the decomposition 
\begin{equation}
\label{eq: decompJoint}
    \begin{pmatrix}
        \hat{\theta}_\delta\\
        \hat{\eta}_\delta
    \end{pmatrix}=\begin{pmatrix}
        {\theta}\\
        {\eta}
    \end{pmatrix}+\norm{K}_{L^2(\R^d)}\mathcal{I}_\delta^{-1}\mathcal{M}_\delta
\end{equation}
on the event $\{\operatorname{det}(\mathcal{I}_\delta)>0\}$ with the martingale part
\begin{equation}
    \label{eq: martingaleJoint}
    \mathcal{M}_\delta=\sum_{k=1}^N\int_0^TY_{\delta,k}(t)\diff W_{x_k}(t).
\end{equation}
As the limiting object of the rescaled observed Fisher information is deterministic and invertible (see \Cref{result: convergence_joint_estimator2} below), $\mathcal{I}_\delta$ will itself be invertible for sufficiently small $\delta$, cf. \cite[Theorem A.7.7, Corollary A.7.8]{küchler_exponential_1997}.
\begin{assumption}[Regularity of the kernel and the initial condition]
\label{ass: mainass2}
        \begin{enumerate}[label=(\roman*)]
        \item[]
        \item \label{assumption:1}The locations $x_k$, $k=1,\dots,N$, belong to a fixed compact set $\mathcal{J}\subset\Lambda$, which is independent of $\delta$ and $N$. There exists $\delta'>0$ such that $\mathrm{supp}(K_{\delta,x_k})\cap\mathrm{supp}(K_{\delta,x_l})=\emptyset$ for $k\neq l$ and all $\delta\leq\delta'$.
        \item \label{assumption:2}There exists a compactly supported function $\tilde K\in H^{\lceil 2\alpha_1\rceil+2\lceil \alpha_1\rceil}(\R^d)$ 
        such that $K=\Delta_0^{\lceil \alpha_1\rceil}\tilde K$.
        \item  \label{assumption:3}The functions $(-\Delta)^{\alpha_i-(\alpha_1+\beta_1)/2}K$ are linearly independent for all $i=1,\dots ,p$, and the functions $(-\Delta)^{\beta_j-\beta_1/2}K$ are linearly independent for all $j=1,\dots,q$.
        \item \label{assumption:4}The initial condition $(u_0,v_0)^\top$ in \eqref{eq: genspde} takes values in $\dot H^{2\alpha_1}(\Lambda)\times \dot H^{\alpha_1}(\Lambda)$.
        \end{enumerate}
\end{assumption}
\Cref{ass: mainass2} (i) ensures that $\sca{K_{\delta,x_k}}{K_{\delta,x_l}}=\norm{K}^2_{L^2(\R^d)}\delta_{k,l}$ with the Kronecker-delta $\delta_{k,l}$. Consequently, the Brownian motions $W_{x_k}$ become mutually independent if $\delta$ is sufficiently small. Thus, $\mathcal{I}_\delta$ forms the quadratic variation process of the time-martingale $\mathcal{M}_\delta$, and we expect $\mathcal{I}_\delta^{-1/2}\mathcal{M}_\delta$ to be asymptotically normally distributed. Both (ii) and (iii) guarantee that the limiting object of the observed Fisher information is well-defined and invertible, while (iv) ensures that the initial condition is asymptotically negligible. In principle, the required smoothness in (ii) can be relaxed, depending on the dimension $d$ and the identifiability of the appearing parameters in \eqref{eq: genspde}, but is kept for the simplification of the proofs. 

We define a diagonal matrix of scaling coefficients $\rho_\delta\in\R^{(p+q)\times (p+q)}$ via
\begin{equation}
\label{eq: rho}
    (\rho_{\delta})_{ii}\coloneqq\begin{cases}
        N^{-1/2}\delta^{2\alpha_i-\alpha_1-\beta_1},\quad 1 \leq i\leq p,\\
        N^{-1/2}\delta^{2\beta_{i-p}-\beta_1},\quad p<i\leq p+q,
    \end{cases}
\end{equation}
and the constant $C(\eta_1,T)$ through
\begin{equation}
\label{eq: constantC}
    C(\eta_1,T)\coloneqq\begin{cases}
        \frac{e^{T\eta_1}-T\eta_1-1}{2\eta_1^2},&\quad\eta_1\neq0,\\
        \frac{T^2}{4},&\quad \eta_1=0.
    \end{cases}
\end{equation}
The following result shows the asymptotic normality of the estimator \eqref{eq: jointest}. 
\begin{theorem}[{Asymptotic behaviour of the joint estimator}]
    \label{result: convergence_joint_estimator2}
    Grant \Cref{ass: mainass2}.
    \begin{itemize}
        \item [(i)] The matrix $\Sigma_{\theta,\eta}\in\R^{(p+q)\times (p+q)}$, given by
        $$\Sigma_{\theta,\eta}\coloneqq \begin{pmatrix}
            \Sigma_{1,\theta,\eta}&0\\
            0&\Sigma_{2,\theta,\eta}
        \end{pmatrix}$$
        with 
        \begin{align*}
            (\Sigma_{1,\theta,\eta})_{ij}&=\begin{cases}
                -\frac{C(\eta_1,T)}{\theta_1}\norm{(-\Delta_0)^{(\alpha_i+\alpha_j-\alpha_1)/2}K}^2_{L^2(\R^d)},&\quad \beta_1=0,\\
                \frac{T}{2\theta_1\eta_1}\norm{(-\Delta_0)^{(\alpha_i+\alpha_j-\alpha_1-\beta_1)/2}K}^2_{L^2(\R^d)},&\quad \beta_1>0,
            \end{cases}\\
            (\Sigma_{2,\theta,\eta})_{kl}&=\begin{cases}
                C(\eta_1,T)\norm{K}^2_{L^2(\R^d)},&\quad \beta_1=0,\\
                -\frac{T}{2\eta_1}\norm{(-\Delta_0)^{(\beta_k+\beta_l-\beta_1)/2}K}^2_{L^2(\R^d)},&\quad \beta_1>0,
            \end{cases}
        \end{align*}
        for $1\leq i,j\leq p$ and $1\leq k,l\leq q$,
        is well-defined and invertible. In particular, the observed Fisher information matrix admits the convergence 
        \begin{equation*}
            \rho_\delta\mathcal{I}_\delta\rho_\delta\stackrel{\mathbb{P}}{\rightarrow}\Sigma_{\theta,\eta}, \quad \delta \rightarrow 0.
        \end{equation*}
        \item [(ii)] The estimator $(\hat{\theta_\delta},\hat{\eta}_\delta)^\top$ is consistent and asymptotically normal, i.e
        $$\rho_\delta^{-1}\begin{pmatrix}
            \hat{\theta}_\delta-\theta\\
            \hat{\eta}_\delta-\eta
        \end{pmatrix}\stackrel{d}{\rightarrow}\mathcal{N}(0,\norm{K}^2_{L^2(\R^d)}\Sigma_{\theta,\eta}^{-1}),\quad\delta\rightarrow0.$$
    \end{itemize}
    
\end{theorem}
The convergence rates among the different parameters are given by \eqref{eq: rho}. As the number of observation points cannot exceed $N\asymp\delta^{-d}$ due to the disjoint support condition of \Cref{ass: mainass2}, not all coefficients can, in general, be consistently estimated in all dimensions, see \Cref{ex: application}. In contrast to parameter estimation in convection-diffusion equations based on local measurements in \cite{altmeyer_anisotrop2021}, the convergence rates for speed parameters are influenced not only by the order $\alpha_1$ of $A_\theta$, but also by $\beta_1$, the order of the damping operator $B_\eta$. Unsurprisingly, higher-order damping results in worse convergence rates as the parameters are harder to identify due to the associated dissipation of energy within the system. On the other hand, the rates for the damping coefficients are not influenced by the order of $A_\theta$, and their rates mirror the rates known from parabolic equations, cf. \cite{huebner_asymptotic_1995,altmeyer_anisotrop2021}. Similar effects were already observed under the full observation scheme $N\asymp\delta^{-d}$ in the spectral approach, cf. \cite{lototsky_multichannel_2010}, leading to identical convergence rates.

In addition to the joint asymptotic normality of the augmented MLE $(\hat{\theta}_\delta,\hat{\eta}_\delta)^\top$, \Cref{result: convergence_joint_estimator2} further yields the asymptotic independence of its components, i.e. the marginal estimators for elastic and damping parameters are asymptotically independent.

If the equation is weakly damped, i.e. if $\beta_1=0$ and $\eta_1<0$, then the term $-T\eta_1$ dominates the expression $(e^{T\eta_1}-T\eta_1-1)/2\eta_1^2$ in the asymptotic variance within \eqref{eq: constantC} as $T \rightarrow \infty$. The converse is true in the amplified case with $\eta_1>0$. If $\eta_1 = 0$, the asymptotic variance of the augmented MLE for $\vartheta$ and $\eta_1$ depends on the time horizon through $T^{-2}$ as discussed in \cite[Remark 5.8]{ziebell_wave23}. It mirrors the rate of convergence of the MLE in the ergodic, stable, and explosive case of the standard Ornstein--Uhlenbeck process as described in \cite[Proposition 3.46]{kutoyants_statistical_2013}.

If, on the other hand, $\beta_1>0$, then the asymptotic variance of the MLE is of order $T^{-1}$ in time. In other words, any dissipation decelerates the temporal convergence rate to the rate $T^{-1}$ associated with parabolic equations.

\begin{remark}[{Parameter estimation under higher-order damping}]
For simplicity, we did not consider cases where the damping dominates \eqref{eq: genspde}, i.e. where $2\beta_1>\alpha_1.$ Nonetheless, studying parameter estimation in those situations is neither impossible nor does it require new approaches. It solely relies on a careful analysis of underlying terms within the asymptotic analysis of the observed Fisher information, which may potentially become complex-valued, cf. also \Cref{rmk: complexOp}. Taking this into account, similar convergence rates may be established.
\end{remark}

\begin{example}
\label{ex: application}
\begin{itemize}
\item[]
\item[(a)] Weakly damped (or amplified) wave equation: Consider the weakly damped (or amplified) wave equation $(\theta_1>0$, $\eta_1\in\R$):
        $$\diff v(t)=(\theta_1\Delta u(t)+\eta_1v(t))\diff t+\diff W(t),\quad 0<t\leq T.$$
        Then, \Cref{result: convergence_joint_estimator2} implies
        \begin{align*}
            \begin{pmatrix}
                N^{-1/2}\delta (\hat{\theta}_\delta-\theta_1)\\
                N^{-1/2}(\hat{\eta}_\delta-\eta_1)
            \end{pmatrix}\stackrel{d}{\rightarrow}\mathcal{N}\left(\begin{pmatrix}
                0\\0
            \end{pmatrix},\begin{pmatrix}
                \frac{\theta_1\norm{K}^2_{L^2(\R^d)}}{C(\eta_1,T)\norm{(-\Delta_0)^{1/2}K}^2_{L^2(\R^d)}}&0\\
                0&\frac{1}{C(\eta_1,T)}
            \end{pmatrix}\right).
        \end{align*}
        Thus, the augmented MLE attains the convergence rate known from the spectral approach, see \cite{lototsky_multichannel_2010} if it is provided with the maximal number of spatial observations $N\asymp\delta^{-d}$. Interestingly, the limiting variance of $\hat{\eta}_\delta$ is independent of the kernel function $K$ similar to the augmented MLE for the first order transport coefficient in \cite{altmeyer_anisotrop2021}.
    \item[(b)] Clamped plate equation: Consider the clamped plate equation with 
        \begin{itemize}
            \item [1)] Weak damping ($\theta_1>0$, $\eta_1\in\R$): 
            \begin{equation}
                \label{eq: plate weak}
                \diff v(t)=(-\theta_1\Delta^2 u(t)+\eta_1v(t))\diff t+\diff W(t),\quad 0<t\leq T.
            \end{equation}
            \item [2)] Structural damping ($\theta_1>0$, $\eta_1>0$): 
            \begin{equation}
                \label{eq: plate struc}
                \diff v(t)=(-\theta_1\Delta^2 u(t)+\eta_1\Delta v(t))\diff t+\diff W(t),\quad 0<t\leq T.
            \end{equation}
        \end{itemize}
        A realisation of the solution can be seen in \Cref{fig: 1}. Depending on the type of damping, the convergence rate for both $\theta_1$ and $\eta_1$ changes. In the case of the weakly damped plate equation \eqref{eq: plate weak}, the CLT yields
        \begin{align*}
            \begin{pmatrix}
                N^{1/2}\delta^{-2} (\hat{\theta}_\delta-\theta_1)\\
                N^{1/2}(\hat{\eta}_\delta-\eta_1)
            \end{pmatrix}\stackrel{d}{\rightarrow}\mathcal{N}\left(\begin{pmatrix}
                0\\0
            \end{pmatrix},\begin{pmatrix}
                \frac{\theta_1\norm{K}^2_{L^2(\R^d)}}{C(\eta_1,T)\norm{\Delta_0K}^2_{L^2(\R^d)}}&0\\
                0&\frac{1}{C(\eta_1,T)}
            \end{pmatrix}\right),
        \end{align*}
        while 
        \begin{align*}
            \begin{pmatrix}
                N^{1/2}\delta^{-1} (\hat{\theta}_\delta-\theta_1)\\
                N^{1/2}\delta^{-1} (\hat{\eta}_\delta-\eta_1)
            \end{pmatrix}\stackrel{d}{\rightarrow}\mathcal{N}\left(\begin{pmatrix}
                0\\0
            \end{pmatrix},\begin{pmatrix}
                \frac{2\theta_1\eta_1\norm{K}^2_{L^2(\R^d)}}{T\norm{(-\Delta_0)^{1/2}K}^2_{L^2(\R^d)}}&0\\
                0&\frac{2\eta_1\norm{K}^2_{L^2(\R^d)}}{T\norm{(-\Delta_0)^{1/2}K}^2_{L^2(\R^d)}}
            \end{pmatrix}\right).
        \end{align*}
        holds under the structural damping given in \eqref{eq: plate struc}. The asymptotic variances between $\hat{\theta}_\delta$ and $\hat{\eta}_\delta$ coincide in the cases $\theta_1=\norm{\Delta_0K}^2_{L^2(\R^d)}\norm{K}_{L^2(\R^d)}^{-2}$ or $\theta_1=1$, respectively. The consistency and the varying convergence rates of the estimators are visualised in \Cref{fig: rates_weak}. Based on the finite difference scheme within the semi-implicit Euler-Maruyama method \cite[Chapter 10]{lord_introduction_2014} (with $10000000\times 2000$ time-space grid points), we computed the root mean squared error (RMSE) for decreasing resolution level $\delta$ from 100 Monte Carlo runs, $N\asymp\delta^{-1}$ measurement locations and the kernel function $K(x)=\exp(-5/(1-x^2))\mathbf{1}(|x|<1)$. In the weakly damped case, it can be seen that the estimator for the elastic coefficient $\theta_1$ achieves a much quicker convergence rate than the estimator of the damping coefficient $\eta_1$. On the other hand, their rates are equal under structural damping. The asymptotic variances are attained in both cases.
    \item[(c)] General hyperbolic equation: Consider the hyperbolic equation
        $$\diff v(t)=\left(\sum_{i=1}^p\theta_i(-\Delta)^{\alpha_i}u(t)+\sum_{j=1}^q\eta_j(-\Delta)^{\beta_j} v(t)\right)\diff t+\diff W(t),\quad 0<t\leq T,$$
        with $p+q$ unknown parameters. Then, the convergence rates for $\theta_i$ and $\eta_j$, respectively, are given by 
        \begin{align*}
            \begin{cases}
                N^{-1/2}\delta^{2\alpha_i-\alpha_1-\beta_1},\quad 1\leq i\leq p,\\
                N^{-1/2}\delta^{2\beta_j-\beta_1},\quad 1\leq j\leq q.
            \end{cases}
        \end{align*}
        Given the maximal number of local measurements $N\asymp\delta^{-d}$, these rates translate to
        \begin{align}
        \label{eq:spectral_rates}
            \begin{cases}
                \delta^{d/2+2\alpha_i-\alpha_1-\beta_1},\quad 1\leq i\leq p,\\
                \delta^{d/2+2\beta_j-\beta_1},\quad 1\leq j\leq q.
            \end{cases}
        \end{align}
        Thus, our method provides a consistent estimator for a parameter $\theta_i$ or $\eta_j$, respectively, if and only if the conditions
        \begin{align}
            \alpha_i&>(\alpha_1+\beta_1-d/2)/2,\quad i\leq p,\label{eq: consis1}\\
            \beta_j&>(\beta_1-d/2)/2,\quad j\leq q,\label{eq: consis2}
        \end{align}
        hold. The convergence rates \eqref{eq:spectral_rates} are also obtained in the spectral regime, cf. \cite[Theorem 1.1 and Theorem 4.1]{lototsky_multichannel_2010} under the consistency conditions \eqref{eq: consis1} and \eqref{eq: consis2}. Given equality in these conditions, the authors further verified a logarithmic rate. We believe that the logarithmic rates in this boundary case are also valid in the local measurement approach given that less restrictive assumptions on the kernel $K$ are imposed, similar to \cite[Proposition 5.4]{altmeyer_anisotrop2021} in a related parabolic problem. 
\end{itemize}
\end{example}
\begin{figure}[t]
    \centering
    \begin{subfigure}{.49\linewidth}
        \centering\includegraphics[width=1.\linewidth]{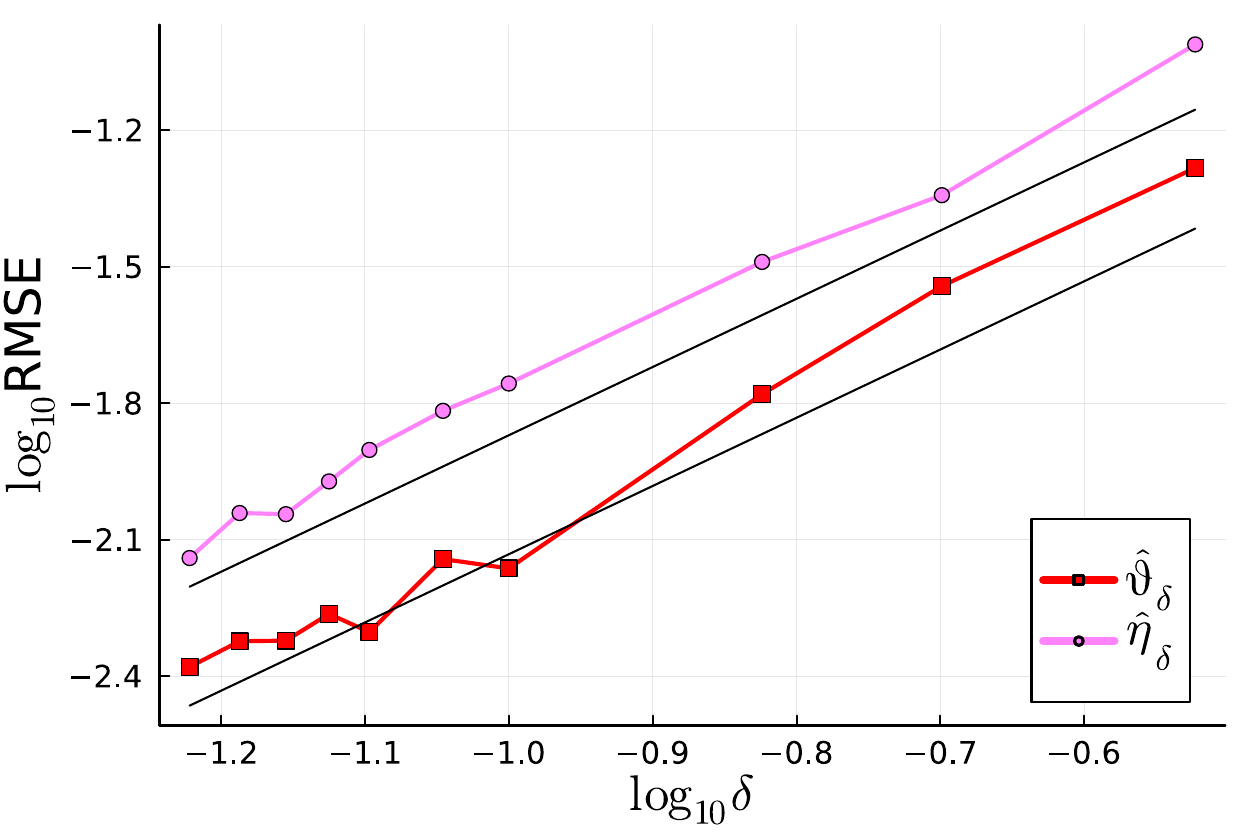}
    \end{subfigure}
    \begin{subfigure}{.49\linewidth}
        \centering\includegraphics[width=1.0\linewidth]{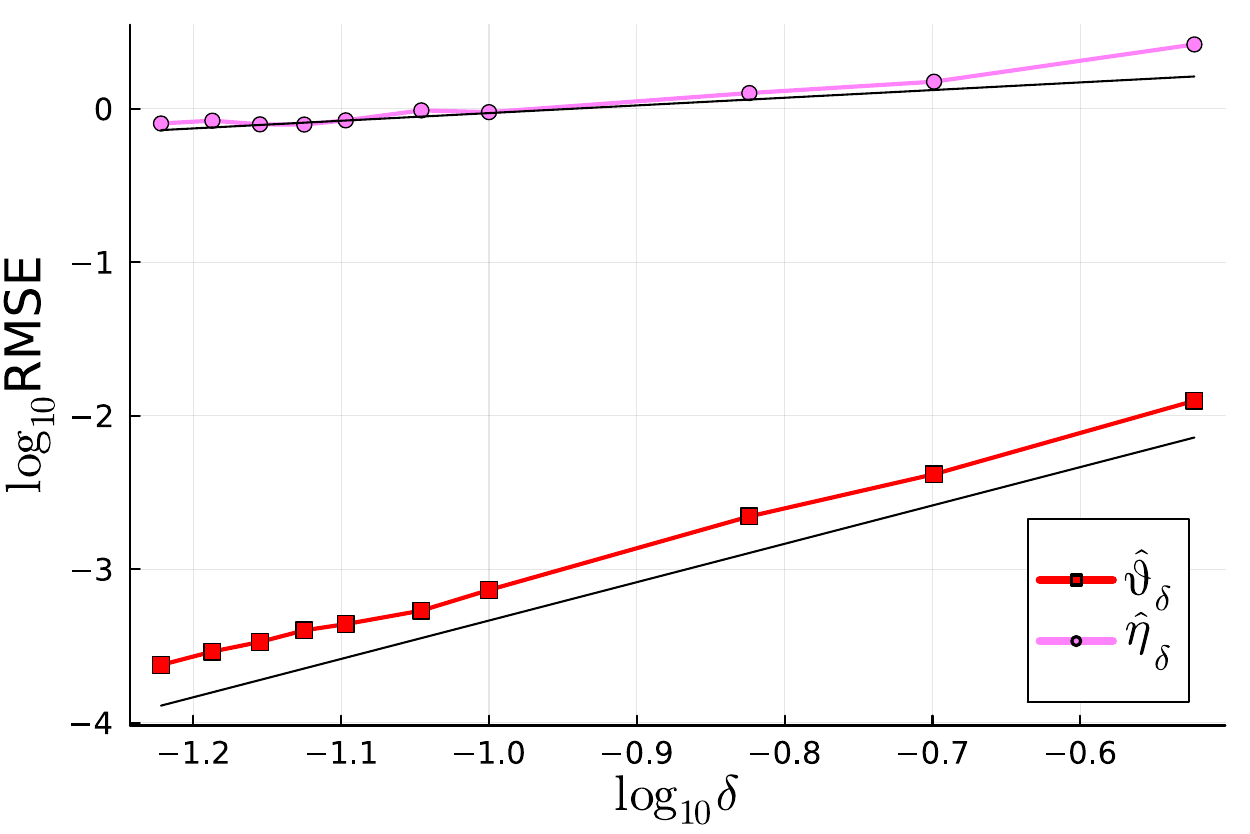}
    \end{subfigure}
    \caption{$\log$-$\log$ plot of the RMSE for $\delta\rightarrow0$ and a maximal number of measurement locations in $d=1$ compared with the theoretical rate in black; (left) structurally damped \eqref{eq: plate struc} with $\eta_1=0.3$, $\theta_1=0.3$; (right) weak damping \eqref{eq: plate weak} with $\eta_1=-0.3,$ $\theta_1=0.3$.}
    \label{fig: rates_weak}
\end{figure}

\section{Proofs}
\label{sec: proofs}
For $\delta>0$ and $x\in\Lambda$ denote by $\Delta_{\delta,x}$ the Laplace operator with Dirichlet boundary conditions on $\Lambda_{\delta,x}$ and define the following differential operators with domain $\dot{H}^{2\alpha_1}(\Lambda)$ and $\dot{H}^{2\alpha_1}(\Lambda_{\delta,x})$, respectively: 
    \begin{align}
        \label{eq: rescaled_operator_gen}
        L_{\theta,\eta} z&\coloneqq(-A_\theta-B_\eta^2/4) z=-\sum_{i=1}^p\theta_i(-\Delta)^{\alpha_i}z-\frac{1}{4}\sum_{k,l=1}^q\eta_k\eta_l(-\Delta)^{\beta_k+\beta_l}z,\nonumber\\
        L_{\theta,\eta,\delta,x}z&\coloneqq-\sum_{i=1}^p\delta^{2\alpha_1-2\alpha_i}\theta_i(-\Delta_{ {\delta,x}})^{\alpha_i}z-\frac{1}{4}\sum_{k,l=1}^q\delta^{2\alpha_1-2\beta_k-2\beta_l}\eta_k\eta_l(-\Delta_{{\delta,x}})^{\beta_k+\beta_l}z\nonumber.
    \end{align}
    Introduce further the rescaled versions of $B_\eta$ and $A_\theta$, defined through
    \begin{align*}
        B_{\eta,\delta,x}&\coloneqq\sum_{j=1}^q\delta^{2\beta_1-2\beta_j}\eta_j(-\Delta_{ {\delta,x}})^{\beta_j},\quad D(B_{\eta,\delta,x})=D((-\Delta_{\delta,x})^{\beta_1})=\dot{H}^{2\beta_1}(\Lambda_{\delta,x}),\\
        A_{\theta,\delta,x}&\coloneqq\sum_{i=1}^p\delta^{2\alpha_1-2\alpha_i}\theta_i(-\Delta_{\delta,x})^{\alpha_i}, \quad D(A_{\theta,\delta,x})=D((-\Delta_{\delta,x})^{\alpha_1})=\dot{H}^{2\alpha_1}(\Lambda_{\delta,x}),
    \end{align*}
    and the limiting objects
    \begin{align*}
        \bar{L}_{\theta,\eta}z&\coloneqq\begin{cases}
            -\theta_1(-\Delta_{0})^{\alpha_1}z,\quad \alpha_1>2\beta_1,\\
            -(\theta_1+\frac{\eta_1^2}{4})(-\Delta_{0})^{\alpha_1}z,\quad \alpha_1=2\beta_1,
        \end{cases} \quad D(\bar{L}_{\theta,\eta})=D((-\Delta_0)^{\alpha_1})=\dot{H}^{2\alpha_1}(\mathbb{R}^{d}),\\
        \bar{B}_\eta&\coloneqq\eta_1(-\Delta_{0})^{\beta_1}, \quad D(\bar{B}_\eta)=D((-\Delta_{0})^{\beta_1})=\dot{H}^{2\beta_1}(\mathbb{R}^{d}),\\
        \bar A_{\theta}&\coloneqq\theta_1(-\Delta_0)^{\alpha_1}, \quad D(\bar A_{\theta})= D((-\Delta_0)^{\alpha_1})=\dot{H}^{2\alpha_1}(\mathbb{R}^{d}).
        \end{align*}
    The study of the operator $L_{\theta,\eta}$ and the convergence $L_{\theta,\eta,\delta,x}\rightarrow \bar{L}_{\theta,\eta}$ is of particular importance as $L_{\theta,\eta}$ appears in the solution to the coupled second-order system \eqref{eq: genspde} through the $M,N$-functions (generalisations of operator cosine and sine functions) introduced in \Cref{lem: MN functions} below. It captures the interaction between the elasticity operator $A_\theta$ and the dissipative operator $B_\eta.$ The asymptotically dominating term of $L_{\theta,\eta,\delta,x}$ determines the limiting object $\bar{L}_{\theta,\eta}$. We will thus start by some fundamental observations regarding the differential operator $L_{\theta,\eta}$ which plays a crucial role in the following \Cref{sec: cosine,sec: asymptotic}.
    
    We will frequently use that $L_{\theta,\eta}$ is an (unbounded) normal operator with spectrum $\sigma(L_{\theta,\eta})$ and the resolution of identity $E$ (cf. \cite[Chapter 13]{rudin_functional_1991}). By the functional calculus for normal unbounded operators, we can define the operator
    $f(L_{\theta,\eta})\coloneqq \int_{\sigma(L_{\theta,\eta})}f(\lambda)\diff E(\lambda)$ on the domain
    \begin{equation*}
        \mathcal{D}_f \coloneqq \mathcal{D}(f(L_{\theta,\eta}))= \left\{z\in L^2(\Lambda):\int_{\sigma(L_{\theta,\eta})}|f(\lambda)|^2\diff E_{z,z}(\lambda)<\infty \right\},
    \end{equation*}
    for any measurable function $f:\mathbb{C}\rightarrow \mathbb{C}$. Analogous statements also apply to $A_\theta$, $B_\eta$ and the rescaled differential operators.
\begin{lemma}[{Rescaling of operators}]
\label{lem: rescaling}
    Let $\delta>0,x\in\Lambda$ and $f:\mathbb{C}\rightarrow\mathbb{C}     \cup\{\pm\infty\}$ be measurable. If $z_{\delta,x}\in\mathcal{D}_f$, then 
    \begin{align*}
    \label{eq: rescaling}
        f(L_{\theta,\eta})z_{\delta,x}&=(f(\delta^{-2\alpha_1}L_{\theta,\eta,\delta,x})z)_{\delta,x},\\
         f(B_\eta)z_{\delta,x}& =(f(\delta^{-2\beta_1}B_{\eta,\delta,x})z)_{\delta,x},\\
         f(A_\theta)z_{\delta,x}&=(f(\delta^{-2\alpha_1}A_{\theta,\delta,x})z)_{\delta,x}.
    \end{align*}
\end{lemma}
\begin{proof}[Proof of \Cref{lem: rescaling}]
    Suppose $z \in \dot{H}^{2\alpha_1}(\Lambda_{\delta,x})$ such that $z_{\delta,x}\in \dot{H}^{2\alpha_1}(\Lambda)\subset \mathcal{D}_f$. Then, the claim follows immediately for $f(x)=x$ by differentiating $z_{\delta,x}$ and from the definition of $L_{\theta,\eta,\delta,x}$, see also \cite[Lemma 16]{altmeyer_parameterSemi_2020}. Using (i) and (iii) of \cite[Proposition 5.15]{schmudgenUnboundedSelfadjointOperators2012}, the result can be extended to measurable $f:\mathbb{C}\rightarrow\mathbb{C} \cup\{\pm\infty\}$ by first passing to the associated resolution of the identities of $L_{\theta,\eta}, B_{\eta}, A_{\vartheta}$ and $L_{\theta,\eta,\delta,x}, B_{\eta, \delta, x}, A_{\vartheta,\delta,x}$ respectively, and interpreting the localisation as a bounded linear operator from $L^2(\Lambda_{\delta,x})$ to $L^2(\Lambda)$.
\end{proof}
Throughout the remainder of the paper, we will assume that $L_{\theta,\eta}$, and thus also $L_{\theta,\eta,\delta,x}$, is a positive operator. A sufficient condition for this is given in the next lemma.
\begin{lemma}[{Sufficient condition for positivity}]
    \label{lem: positiveoperator} Let $(e_k)_{k\in\N}$ form an orthonormal basis of $(-\Delta)$ with eigenvalues $\lambda_k\geq c(\Lambda)$ for some constant $c(\Lambda)>0.$ Then, $L_{\theta,\eta}$ is a positive operator if one of the following conditions is satisfied:
    \begin{itemize}
        \item [(i)] \Cref{ass: parameterass} holds, $c(\Lambda)\geq1$ and 
        $$|\theta_1|>\sum_{i=2}^p|\theta_i|+\frac{1}{4}\sum_{k,l=1}^q|\eta_k\eta_l|;$$
        \item [(ii)] \Cref{ass: parameterass} holds, $c(\Lambda)<1$ and 
        $$|\theta_1|>\sum_{i=2}^p|\theta_i|c(\Lambda)^{\alpha_i-\alpha_1}+\frac{1}{4}\sum_{k,l=1}^q|\eta_k\eta_l|c(\Lambda)^{\beta_k+\beta_l-\alpha_1}.$$
    \end{itemize}
\end{lemma}
\begin{proof}[Proof of \Cref{lem: positiveoperator}]
    \Cref{ass: parameterass} states in particular that $\theta_1<0$ and, additionally, $\theta_1+\eta_1^2/4<0$ in case $\alpha_1=2\beta_1.$ It is now enough to show that all eigenvalues of $L_{\theta,\eta}$ are positive, which holds if for all $x\geq c(\Lambda)$:
    \begin{equation}
    \label{eq: positivity}
        |\theta_1|x^{\alpha_1}-\sum_{i=2}^p|\theta_i|x^{\alpha_i}-\frac{1}{4}\sum_{k,l=1}^q|\eta_k\eta_l|x^{\beta_k+\beta_l}>0.
    \end{equation}
    \begin{itemize}
        \item [(i)] If $c(\Lambda)\geq1$, then both $x^{\alpha_i}$ and $x^{\beta_k+\beta_l}$ are bounded by $x^{\alpha_1}$ for any $i\leq p$ and $k,l\leq q,$ thus \eqref{eq: positivity} is satisfied.
        \item [(ii)] If $c(\Lambda)<1$, then $x^{\alpha_i-\alpha_1}\leq c(\Lambda)^{\alpha_i-\alpha_1}$ and $x^{\beta_k+\beta_l-\alpha_1}\leq c(\Lambda)^{\beta_k+\beta_l-\alpha_1}$. \qedhere
    \end{itemize}
\end{proof}
\begin{remark}
\label{rmk: complexOp}
    If $L_{\theta,\eta}$ is not a positive operator and has non-positive eigenvalues, any choice of the operator root is a complex-valued operator. This is particularly the case if the structural damping exceeds the threshold $\alpha_1=2\beta_1$. Consequently, the associated family of $M, N$-functions in the following subsection is again complex-valued. Thus, inner products hereinafter are associated with complex Hilbert spaces.  However, this does not influence the asymptotic results of \Cref{section:asymptotic_results} due to the convergence $L_{\theta,\eta,\delta,x}\rightarrow\bar{L}_{\theta,\eta}$ to a positive limiting operator $\bar L_{\theta,\eta}$.
\end{remark}

\subsection{Properties of generalised cosine and sine operator functions}
\label{sec: cosine}
\begin{lemma}[{Representations $M,N$-functions}]
    \label{lem: MN functions}
    The operators $A_\theta$ and $B_\eta$ defined in \eqref{eq: opBA} generate a family of $M,N$-functions $(M(t), N(t), t \geq 0)$ given by 
    \begin{align}
        \label{eq: M}M(t)&\coloneqq \mathbf{m}_t(L_{\theta,\eta},B_\eta)\coloneqq e^{B_\eta t/2}\big(\cos{(L_{\theta,\eta}^{1/2}t)}-\frac{B_\eta}{2}\sin{(L_{\theta,\eta}^{1/2}t})L_{\theta,\eta}^{-1/2}\big),\quad L^2(\Lambda)\subset\mathcal{D}(N(t)),\\
        \label{eq: N}N(t)&\coloneqq \mathbf{n}_t(L_{\theta,\eta},B_\eta) \coloneqq e^{B_\eta t/2}\sin{(L_{\theta,\eta}^{1/2}t})L_{\theta,\eta}^{-1/2}, \quad L^2(\Lambda)\subset\mathcal{D}(N(t)).
    \end{align}
\begin{proof}[Proof of \Cref{lem: MN functions}]
    Note that by \cite[Theorem 5.9]{schmudgenUnboundedSelfadjointOperators2012} all of the appearing operators $e^{B_\eta t/2}$, $\cos(tL_{\theta,\eta}^{1/2})$, $\sin(tL_{\theta,\eta}^{1/2})$, $B_\eta$ and $L_{\theta,\eta}^{-1/2}$ in \eqref{eq: M} and \eqref{eq: N} are well-defined and even commute on the smallest occurring domain as they are all based on the same underlying Laplace operator. By direct computation, one can now verify that the conditions \textbf{(M1)-(M4)} in \cite[Definition 1.7.2]{melnikova_abstract_2001} are satisfied by $M(t)$ and $N(t)$ from \eqref{eq: M} and \eqref{eq: N} using the functional calculus. 
    \end{proof}
\end{lemma}
\begin{lemma}[{Self-adjointness of $M,N$-functions}]
\label{lem: self_adjointness_MN}
Assume that $L_{\theta,\eta}$ is a positive operator. Then, the $M, N$-functions defined through \eqref{eq: M} and \eqref{eq: N} are self-adjoint.  
\end{lemma}
\begin{proof}[Proof of \Cref{lem: self_adjointness_MN}] 
The unique positive self-adjoint operator root of the positive self-adjoint operator $L_{\vartheta,\eta}$ is well-defined and exists by \cite[Proposition 5.13]{schmudgenUnboundedSelfadjointOperators2012}. Thus, in view of \Cref{lem: MN functions}, the $M, N$-functions can each be interpreted as the applications of a real-valued function to the underlying Laplace operator on a bounded spatial domain. In particular, by \cite[Theorem 5.9]{schmudgenUnboundedSelfadjointOperators2012} the $M,N$-functions are self-adjoint. 
\end{proof}
As we are interested in the effect of $M, N$-functions applied to localised functions, we further define the rescaled $M, N$-functions:
\begin{align}
    \label{eq: Mres}
    M_{\delta,x}(t)&\coloneqq\mathbf{m}_t(\delta^{-2\alpha_1}L_{\theta,\eta,\delta,x},\delta^{-2\beta_1}B_{\eta,\delta,x}),\\
    \label{eq: Nres}
    N_{\delta,x}(t)&\coloneqq\mathbf{n}_t(\delta^{-2\alpha_1}L_{\theta,\eta,\delta,x},\delta^{-2\beta_1}B_{\eta,\delta,x}).
\end{align}
An application of Lemma \ref{lem: rescaling} yields the scaling properties of the $M,N$-functions in analogy to \cite[Lemma 3.1]{altmeyer_nonparametric_2020} and \cite[Lemma 3.1]{ziebell_wave23}:
\begin{equation*}
M(t)z_{\delta,x}=(M_{\delta,x}(t)z)_{\delta,x},\quad N(t)z_{\delta,x}=(N_{\delta,x}(t)z)_{\delta,x}, \quad z\in L^2(\Lambda_{\delta,x}).
\end{equation*}
\begin{lemma}[{Semigroup upper bounds}]
\label{lem: semigroupbound}
Let $0\leq t\leq T\delta^{-2\beta_1}$, $\gamma\geq0$ and $z\in H^{2\lceil\gamma\rceil}(\R^d)$ with compact support in $\bigcap_{x\in\mathcal{J}}\Lambda_{\delta,x}$ and such that there exists a compactly supported function $\tilde{z}\in H^{2\lceil\gamma\rceil+\lceil\alpha_1\rceil}(\R^d)$ with $z=\Delta^{\lceil\alpha_1\rceil}_0\tilde z$. Then, if $\beta_1>0$, we have
\begin{align}
    \sup_{x\in\mathcal{J}}\norm{e^{tB_{\eta,\delta,x}}L_{\theta,\eta,\delta,x}^{-1/2}(-\Delta_{\delta,x})^\gamma z}_{L^2(\Lambda_{\delta,x})}&\lesssim 1\wedge t^{-(\gamma+\lceil\alpha_1\rceil-\alpha_1/2)/\beta_1};\label{eq: semi1bound}\\
    \sup_{x\in\mathcal{J}}\norm{e^{tB_{\eta,\delta,x}}(-\Delta_{\delta,x})^\gamma z}_{L^2(\Lambda_{\delta,x})}&\lesssim 1\wedge t^{-(\gamma+\lceil\alpha_1\rceil)/\beta_1}\label{eq: semi2bound}.
\end{align}
Moreover, in case $\beta_1=0$, the left-hand sides in \eqref{eq: semi1bound} and \eqref{eq: semi2bound} are bounded by a constant independent of $\delta$ and $t$.
\end{lemma}
\begin{proof}[Proof of \Cref{lem: semigroupbound}]
The key idea of the proof is that all involved operators emerge as an application of the functional calculus applied to the same Laplace operator. In particular, they are simultaneously diagonalisable through the same eigenfunctions. Note that in contrast to the eigenfunctions, the associated eigenvalues do not depend themselves on the shift, i.e. $x \in \Lambda$, within the rescaling of the Laplace operator.  

Let $\beta_1>0.$ We only prove \eqref{eq: semi2bound}, since the argument for \eqref{eq: semi1bound} is similar, using additionally that $L_{\theta,\eta,\delta,x}$ commutes with $(-\Delta_{\delta,x})^\gamma$, $\operatorname{ord}(L_{\theta,\eta,\delta,x})=2\alpha_1$ and a bound of $L_{\theta,\eta,\delta,x}$ in terms of its leading term $(-\Delta_{\delta,x})^{\alpha_1}.$ Let $(e_k)_{k\in\N}$ form an orthonormal basis of $(-\Delta)$ in $L^2(\Lambda)$ with eigenvalues $\lambda_k>0$. Then, there exists a constant $c(\Lambda)$ such that $\lambda_k\geq c(\Lambda)$ for all $k\geq1$, see \cite[Proposition 5.2 and Corollary 5.3]{shimakura_partial_1992}. 
We consider the most involved case, that is, $\eta_1<0$ and $\eta_2,\dots,\eta_q>0$. Consequently, $B_\eta$ will, in general, not be a negative operator, but there exists $y_0>0$ such that for all $y\geq y_0$, we have
    $$\eta_1y^{\beta_1}+\sum_{j=2}^q\eta_jy^{\beta_j}\leq \frac{\eta_1y^{\beta_1}}{2},$$
    and all but finitely many eigenvalues of $B_\eta$ will be negative due to  \Cref{ass: parameterass} (ii). Consider the polynomial
    $$P_\eta(y)\coloneqq\frac{\eta_1}{2}y^{\beta_1}+\sum_{j=2}^q\eta_jy^{\beta_j}$$
    and define $C_1\coloneqq\max_{y\in[c(\Lambda),y_0]}|P_\eta(y)|$. Then $$\eta_1y^{\beta_1}+\sum_{j=2}^q\eta_jy^{\beta_j}-C_1\leq \frac{\eta_1y^{\beta_1}}{2}$$
    holds for all $y\geq c(\Lambda)$, and all eigenvalues of the operator $B_\eta-C_1$ are negative and upper bounded by $\eta_1c(\Lambda)/2$. Analogously, $(e_{k,\delta,x})_{k\in\N}$ forms an orthonormal basis of $(-\Delta_{\delta,x})$ in $L^2(\Lambda_{\delta,x})$ with eigenvalues $\lambda_{k,\delta,x} =\delta^2\lambda_k\geq \delta^2c(\Lambda)$. Similar calculations imply that 
    \begin{equation}
        \eta_1y^{\beta_1}+\sum_{j=2}^q\eta_j\delta^{2(\beta_1-\beta_j)}y^{\beta_j}-\delta^{2\beta_1}C_1\leq \frac{\eta_1y^{\beta_1}}{2},\quad y\geq \delta^2c(\Lambda).
    \end{equation}
    Thus, all eigenvalues of the operator difference $B_{\eta,\delta,x}-\delta^{2\beta_1}C_1$ are negative and the difference is bounded by the differential operator $\eta_1(-\Delta_{\delta,x})^{\beta_1}/2$ in the sense that 
    $$\norm{e^{t(B_{\eta,\delta,x}-\delta^{2\beta_1}C_1)}w}_{L^2(\Lambda_{\delta,x})}\leq \norm{e^{t\eta_1(-\Delta_{\delta,x})^{\beta_1}/2}w}_{L^2(\Lambda_{\delta,x})},\quad w\in L^2(\Lambda_{\delta,x}),$$
    independent of $x\in\mathcal{J}.$ Note further that $z=\Delta_0^{\lceil\alpha_1\rceil}\tilde z=\Delta_{\delta,x}^{\lceil\alpha_1\rceil}\tilde z$ since $z$ is compactly supported in $\bigcap_{x\in\mathcal{J}}\Lambda_{\delta,x}$.
    With that, we have all the ingredients to prove \eqref{eq: semi2bound}. For $0\leq t\leq T\delta^{-2\beta_1}$, we obtain
    \begin{align*}
        \sup_{x\in\mathcal{J}}\norm{e^{tB_{\eta,\delta,x}}(-\Delta_{\delta,x})^\gamma z}_{L^2(\Lambda_{\delta,x})}&=\sup_{x\in\mathcal{J}}\norm{e^{t\delta^{2\beta_1}C_1}e^{t(B_{\eta,\delta,x}-\delta^{2\beta_1}C_1)}(-\Delta_{\delta,x})^\gamma z}_{L^2(\Lambda_{\delta,x})}\\
        &\leq e^{C_1T}\sup_{x\in\mathcal{J}}\norm{e^{t(B_{\eta,\delta,x}-\delta^{2\beta_1}C_1)}(-\Delta_{\delta,x})^\gamma (-\Delta_{\delta,x})^{\lceil\alpha_1\rceil}\tilde z}_{L^2(\Lambda_{\delta,x})}\\
        &\leq e^{C_1T}\sup_{x\in\mathcal{J}}\norm{e^{t\eta_1(-\Delta_{\delta,x})^{\beta_1}/2}(-\Delta_{\delta,x})^{\gamma+\lceil\alpha_1\rceil} \tilde z}_{L^2(\Lambda_{\delta,x})}\\
        &\lesssim (1\wedge t^{-(\gamma+\lceil\alpha_1\rceil)/\beta_1})\sup_{x\in\mathcal{J}}\big(\norm{\tilde z}_{L^2(\Lambda_{\delta,x})}+\norm{(-\Delta_{\delta,x})^{\gamma} z}_{L^2(\Lambda_{\delta,x})}\big),
    \end{align*}
    where the last line follows from the fact that $(-\Delta_{\delta,x})^{\beta_1}\eta_1/2$ generates a contraction semigroup and the smoothing property of semigroups. As $\Delta^n_{\delta,x}z=\Delta^n_0z$ holds for any $x\in\mathcal{J}$ and $1\leq n\leq\lceil\gamma\rceil$, an application of the functional calculus yields
    \begin{align*}
        \sup_{x\in\mathcal{J}}\norm{(-\Delta_{\delta,x})^{\gamma}z}_{L^2(\Lambda_{\delta,x})}^2&\leq \sup_{x\in\mathcal{J}}\norm{(-\Delta_{\delta,x})^{\lceil\gamma\rceil}z}_{L^2(\R^d)}^2+\sup_{x\in\mathcal{J}}\norm{(-\Delta_{\delta,x})^{\lfloor\gamma\rfloor}z}_{L^2(\R^d)}^2\\
        &=\norm{(-\Delta_0)^{\lceil\gamma\rceil}z}_{L^2(\R^d)}^2+\norm{(-\Delta_0)^{\lfloor\gamma\rfloor}z}_{L^2(\R^d)}^2\\
        &<\infty,
    \end{align*}
    proving the assertion. The claim for $\beta_1=0$ follows directly by bounding $$\lvert e^{tB_{\eta,\delta,x}}\rvert=\lvert e^{t\eta_1}\rvert\leq C,\quad 0\leq t\leq T,$$
    for some constant $C$ only depending on $\eta_1$ and $T$.
\end{proof}

The theory of cosine operator functions was developed by \cite{sova1966cosine} and led to general deterministic solution theory for undamped second-order abstract Cauchy problems. By substituting the time derivative as its own variable, it is possible to rewrite a second-order abstract Cauchy problem as a first-order abstract Cauchy problem in two components. The associated strongly continuous semigroup then lives on a product of Hilbert spaces called the phase-space; see \cite[Chapter 3.14]{arendtVectorvaluedLaplaceTransforms2001}, \cite{pandolfiSystemsPersistentMemory2021a} and \cite[Chapter 0.3]{melnikova_abstract_2001}. The same remains true under suitable assumptions on the elastic and dissipation operator within a damped abstract second-order Cauchy problem \cite[Chapter 1.7]{melnikova_abstract_2001}.
\begin{lemma}[{Semigroup on the phase-space}]
    \label{lem: generated semigroup}
   The operator $\mathcal{A}_{\theta,\eta}$, defined through
    \begin{equation*}
        \mathcal{A}_{\theta,\eta}\coloneqq\begin{pmatrix}
            0&I\\
            A_\theta&B_\eta
        \end{pmatrix}, \quad D(\mathcal{A_{\vartheta,\eta}})=\dot H^{2\alpha_1}(\Lambda) \times \dot H^{2\beta_1}(\Lambda),
    \end{equation*}
    generates a $C_0$-semigroup $(J_{\theta,\eta}(t))_{t\geq0}$ on the phase-space $\dot H^{\alpha_1}(\Lambda) \times L^{2}(\Lambda)$ given by 
    \begin{equation}
    \label{eq: semigroup}
        J_{\theta,\eta}(t)\coloneqq\begin{pmatrix}
        M(t)&N(t)\\
        A_\theta N(t)&M(t)+B_\eta N(t)
    \end{pmatrix}=\begin{pmatrix}
        M(t)&N(t)\\
        M^\prime(t)&N^\prime(t)
    \end{pmatrix}, \quad t \geq 0,
    \end{equation}
    with $M(t)$ and $N(t)$ given in \eqref{eq: M} and \eqref{eq: N}, respectively.
\end{lemma}
\begin{proof}[Proof of \Cref{lem: generated semigroup}]
    It is well-known that in the special case where both $A_\theta$ and $B_\eta$ are strictly negative operators $\mathcal{A}_{\theta,\eta}$ generates a $C_0$-semigroup, which is even analytic if and only if $\alpha_1/2\leq\beta_1\leq\alpha_1,$ cf. \cite{chen_damping_1989}. On the other hand, $(J_{\theta,\eta}(t))_{t\geq0}$ given by \eqref{eq: semigroup} is indeed a semigroup generated by $\mathcal{A}_{\theta,\eta}$ which follows by direct verification of the differential properties of $M,N$-functions in \cite[p. 131]{melnikova_abstract_2001} using the functional calculus.
\end{proof}
The coupled second-order system \eqref{eq: genspde} can also be written as a first-order system 
\begin{equation*}
\diff X(t)=\mathcal{A}_{\theta,\eta}X(t)\diff t+\begin{pmatrix}
    0\\
    I
\end{pmatrix}\diff W(t), \quad 0< t\leq T,
\end{equation*}
for $X(t)=(u(t),v(t))^\top$ and the matrix-valued differential operator $\mathcal{A}_{\theta,\eta}$ generating the strongly continuous semigroup $(J_{\theta,\eta}(t))_{t\geq0}$ constituted by the $M,N$-functions defined in \Cref{lem: MN functions}.
The $M, N$-functions correspond to the cosine and sine functions in the undamped wave equation, see \cite[Chapter 3]{arendtVectorvaluedLaplaceTransforms2001}. Naturally, they appear in the solution to the stochastic partial differential equation \eqref{eq: genspde}:
\begin{align}
\label{eq: solution}
    \begin{pmatrix}
        u(t)\\
        v(t)
    \end{pmatrix}&=J_{\theta,\eta}(t)\begin{pmatrix}u(0)\\
    v(0)
    \end{pmatrix}+\int_0^tJ_{\theta,\eta}(t-s)\begin{pmatrix}
        0\\
        I
    \end{pmatrix}\diff W(s)\\
    &=\begin{pmatrix}
        M(t)u_0+N(t)v_0\\
        M^\prime(t)u_0+N^\prime(t)v_0
    \end{pmatrix}+\begin{pmatrix}
        \int_0^tN(t-s)\diff W(s)\\
        \int_0^tN^\prime(t-s)\diff W(s)
    \end{pmatrix}.\nonumber
\end{align}
\subsection{Asymptotic properties of local measurements}
\label{sec: asymptotic}
In this section, we study the asymptotic covariance structure of the local measurements, which is crucial in showing the convergence of the observed Fisher information matrix $\mathcal{I}_\delta.$ 
\label{section:asymptotic_results}
    \begin{lemma}[{Covariance structure}]
\label{lemma:covariance_structure}
Assume that $(u_0,v_0)^\top=(0,0)^\top.$ For any $t,s \in [0,T]$, $x\in \Lambda$, $1 \leq i,j\leq p$, $1 \leq k,l\leq q$, the covariance between local measurements is given by 
\begin{align*}
        &\mathrm{Cov}(u_{\delta,x}^{\Delta_i}(t), u_{\delta,x}^{\Delta_j}(s)) \\
        &\qquad= \delta^{-2\alpha_i-2\alpha_j}\int_0^{t \land s} \langle N_{\delta,x}(t-r)(-\Delta_{\delta,x})^{\alpha_i} K,N_{\delta,x}(s-r)(-\Delta_{\delta,x})^{\alpha_j} K \rangle_{L^{2}(\Lambda_{\delta,x})}\mathrm{d}r,\nonumber\\
        &\mathrm{Cov}(v_{\delta,x}^{\Delta_k}(t), v_{\delta,x}^{\Delta_l}(s))\\ &\qquad= \delta^{-2\beta_k-2\beta_l}\int_0^{t \land s} \sca{N^\prime_{\delta,x}(t-r)(-\Delta_{\delta,x})^{\beta_k}K}{N^\prime_{\delta,x}(s-r)(-\Delta_{\delta,x})^{\beta_l}K}_{L^2(\Lambda_{\delta,x})}\diff r,\nonumber\\
        &\mathrm{Cov}(u_{\delta,x}^{\Delta_i}(t), v_{\delta,x}^{\Delta_k}(s))\\
        &\qquad=\delta^{-2\alpha_i-2\beta_k}\int_{0}^{t \land s}\langle N_{\delta,x}(t-r)(-\Delta_{\delta,x})^{\alpha_i} K,N^\prime_{\delta,x}(s-r)(-\Delta_{\delta,x})^{\beta_k}K\rangle_{L^{2}(\Lambda_{\delta,x})}\mathrm{d}r.\nonumber
\end{align*}
\end{lemma}
\begin{proof}[Proof of \Cref{lemma:covariance_structure}]
Using \eqref{eq: solution} and \cite[Proposition 4.28]{da_prato_stochastic_2014}, we observe
\begin{align*}
        \mathrm{Cov}(u_{\delta,x}^{\Delta_i}(t), u_{\delta,x}^{\Delta_j}(s)) &=\int_{0}^{t \land s} \langle N^{*}(t-r)(-\Delta)^{\alpha_i} K_{\delta,x},N^{*}(s-r)(-\Delta)^{\alpha_j} K_{\delta,x} \rangle \mathrm{d}r,\\
        \mathrm{Cov}(v_{\delta,x}^{\Delta_k}(t), v_{\delta,x}^{\Delta_l}(s)) &= \int_{0}^{t \land s} \langle (N^\prime(t-r))^*(-\Delta)^{\beta_k}K_{\delta, x},(N^\prime(s-r))^*(-\Delta)^{\beta_l}K_{\delta, x}\rangle \mathrm{d}r,\nonumber\\
        \mathrm{Cov}(u_{\delta,x}^{\Delta_i}(t), v_{\delta,x}^{\Delta_k}(s)) &= \int_{0}^{t \land s} \langle N^{*}(t-r)(-\Delta)^{\alpha_i} K_{\delta,x},(N^\prime(s-r))^*(-\Delta)^{\beta_k}K_{\delta, x}\rangle \mathrm{d}r.\nonumber
\end{align*}
We can rewrite the last equations through the functional calculus by using \Cref{lem: rescaling}, the representations \eqref{eq: Mres} and \eqref{eq: Nres}, self-adjointness by \Cref{lem: self_adjointness_MN} as well as \cite[Theorem 5.9] {schmudgenUnboundedSelfadjointOperators2012}:
\begin{align*}
        &\mathrm{Cov}(u_{\delta,x}^{\Delta_i}(t), u_{\delta,x}^{\Delta_j}(s)) \\
        &\qquad= \delta^{-2\alpha_i-2\alpha_j}\int_0^{t \land s} \langle N_{\delta,x}(t-r)(-\Delta_{\delta,x})^{\alpha_i} K,N_{\delta,x}(s-r)(-\Delta_{\delta,x})^{\alpha_j} K \rangle_{L^{2}(\Lambda_{\delta,x})}\mathrm{d}r,\nonumber\\
        &\mathrm{Cov}(v_{\delta,x}^{\Delta_k}(t), v_{\delta,x}^{\Delta_l}(s))\\ &\qquad= \delta^{-2\beta_k-2\beta_l}\int_0^{t \land s} \sca{N^\prime_{\delta,x}(t-r)(-\Delta_{\delta,x})^{\beta_k}K}{N^\prime_{\delta,x}(s-r)(-\Delta_{\delta,x})^{\beta_l}K}_{L^2(\Lambda_{\delta,x})}\diff r,\nonumber\\
        &\mathrm{Cov}(u_{\delta,x}^{\Delta_i}(t), v_{\delta,x}^{\Delta_k}(s))\\
        &\qquad=\delta^{-2\alpha_i-2\beta_k}\int_{0}^{t \land s}\langle N_{\delta,x}(t-r)(-\Delta_{\delta,x})^{\alpha_i} K,N^\prime_{\delta,x}(s-r)(-\Delta_{\delta,x})^{\beta_k}K\rangle_{L^{2}(\Lambda_{\delta,x})}\mathrm{d}r.\qedhere\nonumber
\end{align*}
\end{proof}

\begin{lemma}[{Scaling limits for $M,N$-functions}]
    \label{lem: scaling limit}  Let $\delta>0.$ Let $z_1,z_2\in L^2(\R^d)$ with compact support in $\bigcap_{x\in\mathcal{J}}\Lambda_{\delta,x}$ such that there exist compactly supported functions $\bar{z}_1,\bar{z}_2\in H^{2\ceil{\alpha_1}}(\R^d)$ with $z_i=\Delta_0^{\ceil{\alpha_1}}\bar{z}_i,$ $i=1,2$. As $\delta \rightarrow 0$, we obtain the following convergences.
    \begin{enumerate}[label=(\roman*)]
            \item[]
            \item Let $t\geq0$. Let $\beta_1=0$, i.e. $B_{\eta}=\eta_1.$ Then, uniformly in $x\in\mathcal{J},$
            \begin{align*}    
                \delta^{-2\alpha_1}\sca{N_{\delta,x}(t)z_1}{N_{\delta,x}(t)z_2}_{L^2(\Lambda_{\delta,x})}&\rightarrow -e^{\eta_1t}\frac{1}{2}\sca{\bar A_\theta^{-1}z_1}{z_2}_{L^2(\R^d)},\\
                \sca{M_{\delta,x}(t)z_1}{M_{\delta,x}(t)z_2}_{L^2(\Lambda_{\delta,x})}&\rightarrow e^{\eta_1t}\frac{1}{2}\sca{z_1}{z_2}_{L^2(\R^d)},\\
                \delta^{-\alpha_1}\sca{N_{\delta,x}(t)z_1}{M_{\delta,x}(t)z_2}_{L^2(\Lambda_{\delta,x})}&\rightarrow 0.
            \end{align*}
            \item Let $r_1\neq r_2$. Let $\beta_1=0.$ Then, uniformly in $x\in\mathcal{J}$, 
            \begin{align*}    
                \delta^{-2\alpha_1}\sca{N_{\delta,x}(r_1)z_1}{N_{\delta,x}(r_2)z_2}_{L^2(\Lambda_{\delta,x})}&\rightarrow 0,\\
                \sca{M_{\delta,x}(r_1)z_1}{M_{\delta,x}(r_2)z_2}_{L^2(\Lambda_{\delta,x})}&\rightarrow 0,\\
                \delta^{-\alpha_1}\sca{N_{\delta,x}(r_1)z_1}{M_{\delta,x}(r_2)z_2}_{L^2(\Lambda_{\Lambda_{\delta,x}})}&\rightarrow 0.
            \end{align*}
        \item Let $0<2\beta_1\leq\alpha_1$, $t\in(0,T].$ Then, uniformly in $x\in\mathcal{J},$
        \begin{align*}
            &\delta^{-2\alpha_1-2\beta_1}\sca{\int_0^{t}N_{\delta,x}(r)^2\diff rz_1}{z_2}_{L^2(\Lambda_{\delta,x})}\rightarrow \frac{1}{2}\sca{\bar B_\eta^{-1}\bar A_\theta^{-1}z_1}{z_2}_{L^2(\R^d)},\\
            &\delta^{-2\beta_1}\sca{\int_0^t(N^\prime_{\delta,x}(r))^2\diff rz_1}{z_2}_{L^2(\Lambda_{\delta,x})}\rightarrow -\frac{1}{2}\sca{\bar{B}_\eta^{-1}z_1}{z_2}_{L^2(\R^d)},\\
            &\delta^{-2\beta_1-\alpha_1}\sca{\int_0^tN_{\delta,x}(r)N^\prime_{\delta,x}(r)\diff rz_1}{z_2}_{L^2(\Lambda_{\delta,x})}\rightarrow0.
        \end{align*}
        \end{enumerate}
        \begin{proof}[Proof of \Cref{lem: scaling limit}] 
        \begin{enumerate}[label=(\roman*)]
            \item []
            \item Using \eqref{eq: N} we have
            \begin{equation*}
                N_{\delta,x}(t)=\delta^{\alpha_1} e^{\eta_1t/2}\sin(t\delta^{-\alpha_1}L_{\theta,\eta,\delta,x}^{1/2})L_{\theta,\eta,\delta,x}^{-1/2}, \quad t \in [0,T],
            \end{equation*}
            and thus
            \begin{align*}
            &\delta^{-2\alpha_1}\sca{N_{\delta,x}(t)z_1}{N_{\delta,x}(t)z_2}_{L^2(\Lambda_{\delta,x})}\\
            &\quad=e^{\eta_1t}\sca{\sin (t\delta^{-\alpha_1}L_{\theta,\eta,\delta,x}^{1/2})L_{\theta,\eta,\delta,x}^{-1/2}z_1}{\sin (t\delta^{-\alpha_1}L_{\theta,\eta,\delta,x}^{1/2})L_{\theta,\eta,\delta,x}^{-1/2}z_2}_{L^2(\Lambda_{\delta,x})}.
            \end{align*}
            Since $S_{\theta,\eta,\delta,x}(t)\coloneqq\sin (tL_{\theta,\eta,\delta,x})L_{\theta,\eta,\delta,x}^{-1/2}$ is the operator sine function, which is generated by $-L_{\theta,\eta,\delta,x}$, and $L_{\theta,\eta,\delta,x}\rightarrow\bar{L}_{\theta,\eta}$ as $\delta \rightarrow 0$, the desired convergence follows by repeating the steps of \cite[Proposition A.10 (i)]{ziebell_wave23} regarding asymptotic equipartition of energy. Note that the assumptions in \cite[Proposition A.10]{ziebell_wave23} can be relaxed, as we are not considering the non-parametric case. The employed strong resolvent convergence and the involved convergence of the spectral measures then follow from \cite[Theorem 1 and 2]{weidmannStrongOperatorConvergence1997} by choosing the core $C_c^{\infty}(\mathbb{R}^{d})$ as described in \cite[Lemma A.6]{ziebell_wave23}. The convergences for the functional calculus associated with the respective spectral measures are then immediate, see \cite[Proposition 3.3]{ziebell_wave23}. Similarly, we observe 
            \begin{align*}
                &\sca{M_{\delta,x}(t)z_1}{M_{\delta,x}(t)z_2}_{L^2(\Lambda_{\delta,x})}\\
                &=e^{\eta_1t}\langle{(\cos (t\delta^{-\alpha_1}L_{\theta,\eta,\delta,x}^{1/2})-\delta^{\alpha_1}\eta_1/2\sin(t\delta^{-\alpha_1}L_{\theta,\eta,\delta,x}^{1/2})L_{\theta,\eta,\delta,x}^{-1/2})z_1},\\
                &\qquad\quad{(\cos (t\delta^{-\alpha_1}L_{\theta,\eta,\delta,x}^{1/2})-\delta^{\alpha_1}\eta_1/2\sin(t\delta^{-\alpha_1}L_{\theta,\eta,\delta,x}^{1/2})L_{\theta,\eta,\delta,x}^{-1/2})z_2}\rangle_{L^2(\Lambda_{\delta,x})}.
            \end{align*}
            Likewise, $C_{\theta,\eta,\delta,x}(t)\coloneqq\cos (tL_{\theta,\eta,\delta,x}^{1/2})$ is the cosine operator function associated with the operator $-L_{\theta,\eta,\delta,x}.$ \cite[Example 3.14.15]{arendtVectorvaluedLaplaceTransforms2001} yields the representation 
            $$C_{\theta, \eta,\delta,x}(t\delta^{-\alpha_1})=\frac{1}{2}\left(U_{\theta,\eta,\delta,x}(t\delta^{-\alpha_1})+U_{\theta,\eta,\delta,x}(-t\delta^{-\alpha_1})\right)$$
            with the unitary group $(U_{\theta,\delta,x}(t))_{t\in\R}$ generated by $i(L_{\theta,\eta,\delta,x}^{1/2})$ on $L^2(\Lambda_{\delta,x})$ and the steps of \cite[Proposition A.10 (i)]{ziebell_wave23} can be repeated to verify convergence. Analogous calculations show
            $$\delta^{-\alpha_1}\sca{N_{\delta,x}(t)z_1}{M_{\delta,x}(t)z_2}_{L^2(\Lambda_{\delta,x})}\rightarrow 0.$$
            All the above convergences hold uniformly in $x \in \mathcal{J}$ since in the parametric case the convergences in \cite[Proposition 4.5 and Lemma A.6]{ziebell_wave23} are uniform in $x \in \mathcal{J}$ when applied to functions with support in $\bigcap_{x\in\mathcal{J}}\Lambda_{\delta,x}$. In fact, restricted to $\bigcap_{x\in\mathcal{J}}\Lambda_{\delta,x},$ the Laplacian $\Delta_{\delta,x}$ is identical to $\Delta_{\delta,y}$ for $y\in\mathcal{J}$ and the associated spectral measures become independent of the spatial point $y \in \mathcal{J}$, when applied to functions with support in $\bigcap_{x\in\mathcal{J}}\Lambda_{\delta,x}$.
            \item The convergences follow similarly to (i) by using the slow-fast orthogonality as presented in \cite[Proposition A.10 (ii)]{ziebell_wave23}.
            \item [(iii)]
            For readability, we suppress various indices throughout the remainder of the proof. Thus, we introduce the following notation:
            \begin{equation}
                \label{eq: indices}A\coloneqq A_{\theta,\delta,x};\quad B\coloneqq B_{\eta,\delta,x};\quad L\coloneqq L_{\theta,\eta,\delta,x};\quad \alpha=\delta^{\alpha_1};\quad \beta=\delta^{\beta_1}.
            \end{equation}
             By definition of $M, N$-functions, substitution and the fundamental theorem of calculus, we then obtain
            \begin{align}
               &\delta^{-2\alpha_1-2\beta_1}\sca{\int_0^{t}N_{\delta,x}(r)^2\diff rz_1}{z_2}_{L^2(\Lambda_{\delta,x})}\nonumber\\
               &=\alpha^{-2}\beta^{-2}\sca{\int_0^{t}N_{\delta,x}(r)^2\diff rz_1}{z_2}_{L^2(\Lambda_{\delta,x})}\nonumber\\
               &=\sca{\int_0^{t\beta^{-2}}e^{rB}\sin^2(r\alpha^{-1}\beta^2L^{1/2})L^{-1}\diff rz_1}{z_2}_{L^2(\Lambda_{\delta,x})}\label{eq: DCTbound}\\
               &=\langle\Big(e^{t\beta^{-2}B}\sin^2(t\alpha^{-1}L^{1/2})B^2-2\alpha^{-1}\beta^2e^{t\beta^{-2}B}BL^{1/2}\cos(t\alpha^{-1}L^{1/2})\sin(t\alpha^{-1}L^{1/2})\nonumber\\
               &\quad\quad\left.+2\alpha^{-2}\beta^4(e^{t\beta^{-2}B}-I)L\right)B^{-1}L^{-1}(4\alpha^{-2}\beta^4L+B^2)^{-1}z_1,z_2\rangle_{L^2(\Lambda_{\delta,x})}.\nonumber
            \end{align}
            We can rewrite the last display as 
            \begin{align}
                &\frac{1}{2}\sca{B^{-1}A^{-1}z_1}{z_2}_{L^2(\Lambda_{\delta,x})}\label{eq: t1}\\
                &-\frac{\alpha^2\beta^{-4}}{4}\sca{e^{t\beta^{-2}B}\sin^2(t\alpha^{-1}L^{1/2})L^{-1}BA^{-1}z_1}{z_2}_{L^2(\Lambda_{\delta,x})}\label{eq: t2}\\
                &+\frac{\alpha\beta^{-2}}{2}\sca{e^{t\beta^{-2}B}\cos(t\alpha^{-1}L^{1/2})\sin(t\alpha^{-1}L^{1/2})L^{-1/2}A^{-1}z_1}{z_2}_{L^2(\Lambda_{\delta,x})}\label{eq: t3}\\
                &-\frac{1}{2}\sca{e^{t\beta^{-2}B}B^{-1}A^{-1}z_1}{z_2}_{L^2(\Lambda_{\delta,x})}\label{eq: t4}.
            \end{align}
            Since $A=A_{\theta,\delta,x}$ converges to $\bar A_{\theta}$, \eqref{eq: t1} converges to $\frac{1}{2}\sca{\bar B_\eta^{-1}\bar A_\theta^{-1}z_1}{z_2}_{L^2(\R^d)}$, while \eqref{eq: t2}, \eqref{eq: t3} and \eqref{eq: t4} tend to zero by the Cauchy-Schwarz inequality and \Cref{lem: semigroupbound}. \par As we will integrate \eqref{eq: DCTbound} on the time interval $[0,T]$ in \Cref{lem: convFisher}, we will already compute a uniform upper bound of $\eqref{eq: DCTbound}$, enabling the usage of the dominated convergence theorem. By the Cauchy-Schwarz inequality and \Cref{lem: semigroupbound} we obtain for a constant $C$ independent of the spatial point $x$ and the resolution level $\delta>0$:
            \begin{equation}
            \label{eq: upperbound DCT}
            \begin{aligned}
                &\sca{\int_0^{t\beta^{-2}}e^{rB}\sin^2(r\alpha^{-1}\beta^2L^{1/2})L^{-1}\diff rz_1}{z_2}_{L^2(\Lambda_{\delta,x})}\\
                &\qquad\leq \int_0^{T\delta^{-2\beta_1}}\norm{e^{rB_{\eta,\delta,x}/2}L_{\theta,\eta,\delta,x}^{-1/2}z_1}_{L^2(\Lambda_{\delta,x})}\norm{e^{rB_{\eta,\delta,x}/2}L_{\theta,\eta,\delta,x}^{-1/2}z_2}_{L^2(\Lambda_{\delta,x})}\diff r\\
                &\qquad\leq \int_0^\infty C\Big(1\wedge r^{-\alpha_1/(2\beta_1)})\Big)^2\diff r\eqqcolon V<\infty.
            \end{aligned}  
            \end{equation}
            Similarly to \eqref{eq: DCTbound}, as $\delta \rightarrow 0$, we obtain 
            \begin{align*}
                &\delta^{-2\beta_1}\sca{\int_0^t(N^\prime_{\delta,x}(r))^2\diff rz_1}{z_2}_{L^2(\Lambda_{\delta,x})}\\
                &=\delta^{-2\beta_1}\sca{\int_0^t\left(M_{\delta,x}(r)+{\delta^{-2\beta_1}}B_{\eta,\delta,x}N_{\delta,x}(r)\right)^2\diff rz_1}{z_2}_{L^2(\Lambda_{\delta,x})}\\
                &=\sca{\int_0^{t\beta^{-2}}e^{rB}\left(\cos(r\alpha^{-1}\beta^2L^{1/2})+\frac{\alpha\beta^{-2}}{2}B\sin(r\alpha^{-1}\beta^2L^{1/2})\right)^2\diff rz_1}{z_2}_{L^2(\Lambda_{\delta,x})}\\
                &=\langle \left(\alpha^2\beta^{-4}Be^{t\beta^{-2}B}\sin^2(t\alpha^{-1}L^{1/2})L^{-1}/4\right.\\
                &\quad\quad+\alpha\beta^{-2}e^{t\beta^{-2}B}\cos(t\alpha^{-1}L^{1/2})\sin(t\alpha^{-1}L^{1/2})L^{-1/2}/2\\
                &\quad\quad\left.+(e^{t\beta^{-2}B}-I)B^{-1}/2\right)z_1,z_2\rangle_{L^2(\Lambda_{\delta,x})}\\
                &\rightarrow-\frac{1}{2}\sca{\bar B_\eta^{-1}z_1}{z_2}_{L^2(\R^d)},
            \end{align*}
            and 
            \begin{align*}
                &\delta^{-2\beta_1-\alpha_1}\sca{\int_0^tN_{\delta,x}(r)N^\prime_{\delta,x}(r)\diff rz_1}{z_2}_{L^2(\Lambda_{\delta,x})}\\
                &=\delta^{-2\beta_1-\alpha_1}\sca{\int_0^tN_{\delta,x}(r)\left(M_{\delta,x}(r)+{\delta^{-2\beta_1}}B_{\eta,\delta,x}N_{\delta,x}(r)\right)\diff rz_1}{z_2}_{L^2(\Lambda_{\delta,x})}\\
                &=\langle\int_0^{t\beta^{-2}}e^{rB}\sin(r\alpha^{-1}\beta^2L^{1/2})L^{-1/2}\\
                &\hspace{10mm}\cdot\left(\cos(r\alpha^{-1}\beta^2L^{1/2})+\frac{\alpha\beta^{-2}}{2}B\sin(r\alpha^{-1}\beta^2L^{1/2})L^{-1/2}\right)\diff rz_1,z_2\rangle_{L^2(\Lambda_{\delta,x})}\\
                &=\frac{1}{2}\sca{\alpha\beta^{-2}e^{t\beta^{-2}B}\sin^2(t\alpha^{-1}L^{1/2})L^{-1}z_1}{z_2}_{L^2(\Lambda_{\delta,x})}\rightarrow0,
            \end{align*}
            again having a uniform upper bound in analogy to \eqref{eq: upperbound DCT}.
        \end{enumerate}
        \end{proof}
\end{lemma}

\begin{lemma}
    \label{lem: convFisher}
    Grant \Cref{ass: mainass2} (i)-(iii) and suppose $(u_0,v_0)^\top=(0,0)^\top$. Recall the definition of $C(\eta_1,T)$ given by \eqref{eq: constantC} and let $\beta_1=0.$ Then, for $1\leq i,j\leq p$ and $1\leq k,l\leq q$, we obtain, as $\delta\rightarrow 0$, the convergences
    \begin{align}
        &\sup_{x\in\mathcal{J}}\left|\delta^{2(\alpha_i+\alpha_j-\alpha_1)}\int_0^T\operatorname{Cov}(u_{\delta,x}^{\Delta_i} (t),u_{\delta,x}^{\Delta_j})\diff t+\frac{C(\eta_1,T)}{\theta_1}\norm{(-\Delta)^{(\alpha_i+\alpha_j-\alpha_1)/2}K}^2_{L^2(\R^d)}\right|\rightarrow0;\label{eq: conv uDelta}\\
        &\sup_{x\in\mathcal{J}}\left|\int_0^T\operatorname{Var}(v_{\delta,x} (t))\diff t-C(\eta_1,T)\norm{K}^2_{L^2(\R^d)}\right|\rightarrow0;\label{eq: conv v}\\
        &\sup_{x\in\mathcal{J}}\left|\delta^{2(\alpha_i-\alpha_1/2)}\int_0^T\operatorname{Cov}(u_{\delta,x}^{\Delta_i} (t),v_{\delta,x}(t))\diff t\right|\rightarrow0.\label{eq: conv cov}
    \end{align}
    If, $0<2\beta_1\leq\alpha_1$ we obtain the convergences
    \begin{align}
        &\sup_{x\in\mathcal{J}}\left|\delta^{2(\alpha_i+\alpha_j-\alpha_1-\beta_1)}\int_0^T\operatorname{Cov}(u_{\delta,x}^{\Delta_i} (t),u_{\delta,x}^{\Delta_j})\diff t-\frac{T}{2\eta_1\theta_1}\norm{(-\Delta)^{(\alpha_i+\alpha_j-\alpha_1-\beta_1)/2}K}_{L^2(\R^d)}^2\right|\rightarrow0;\label{eq: conv uDelta_damp}\\
        &\sup_{x\in\mathcal{J}}\left|\delta^{2(\beta_k+\beta_l-\beta_1)}\int_0^T\operatorname{Cov}(v_{\delta,x}^{\Delta_k} (t),v_{\delta,x}^{\Delta_l}(t))\diff t+\frac{T}{2\eta_1}\norm{(-\Delta)^{(\beta_k+\beta_l-\beta_1)/2}K}_{L^2(\R^d)}^2\right|\rightarrow0;\label{eq: conv v_damp}\\
        &\sup_{x\in\mathcal{J}}\left|\delta^{2(\alpha_i+\beta_k-\beta_1-\alpha_1/2)}\int_0^T\operatorname{Cov}(u_{\delta,x}^{\Delta_i} (t),v_{\delta,x}^{\Delta_k}(t))\diff t\right|\rightarrow0.\label{eq: conv cov_damp}
    \end{align}
\end{lemma}
\begin{proof}[Proof of \Cref{lem: convFisher}]
        With the majorant constructed in \eqref{eq: upperbound DCT} in case of \eqref{eq: conv uDelta_damp} or directly by \Cref{lem: semigroupbound} for \eqref{eq: conv uDelta}, we obtain uniformly in $x\in\mathcal{J}$ by \Cref{lemma:covariance_structure}, \Cref{lem: scaling limit}(i),(iii) and the dominated convergence theorem:
        \begin{align*}
            &\delta^{2(\alpha_i+\alpha_j-\alpha_1-\beta_1)}\int_0^T\operatorname{Cov}(u_{\delta,x}^{\Delta_i} (t),u_{\delta,x}^{\Delta_j}(t))\diff t\\
            &=\delta^{-2\alpha_1-2\beta_1}\int_0^T\int_0^t\sca{N_{\delta,x}(r)^2(-\Delta_0)^{\alpha_i}K}{(-\Delta_0)^{\alpha_j}K}_{L^2(\Lambda_{\delta,x})}\diff r\diff t\\
            &=\begin{cases}
                -\int_0^T\int_0^te^{\eta_1r}\frac{1}{2}\sca{\bar A_\theta^{-1}(-\Delta_0)^{\alpha_i}K}{(-\Delta_0)^{\alpha_j}K}_{L^2(\R^d)}\diff r\diff t+o(1),\quad \beta_1=0,\\
                \int_0^T\frac{1}{2}\sca{\bar B_\eta^{-1}\bar A_\theta^{-1}(-\Delta_0)^{\alpha_i}K}{(-\Delta_0)^{\alpha_j}K}_{L^2(\R^d)}\diff t+o(1),\quad \beta_1>0,
            \end{cases}\\
            &=\begin{cases}
                -\frac{C(\eta_1,T)}{\theta_1}\norm{(-\Delta_0)^{(\alpha_i+\alpha_j-\alpha_1)/2}K}_{L^2(\R^d)}^2+o(1),\quad \beta_1=0,\\
                \frac{T}{2\eta_1\theta_1}\norm{(-\Delta_0)^{(\alpha_i+\alpha_j-\alpha_1-\beta_1)/2}K}_{L^2(\R^d)}^2+o(1),\quad \beta_1>0.
            \end{cases}
        \end{align*}
        This proves \eqref{eq: conv uDelta} and \eqref{eq: conv uDelta_damp}. Analogously, \eqref{eq: conv v} and \eqref{eq: conv v_damp} as well as \eqref{eq: conv cov} and \eqref{eq: conv cov_damp} follow using the remaining convergences in \Cref{lem: scaling limit}.
\end{proof}

\begin{lemma}
    \label{lem: convVar}
    Grant \Cref{ass: mainass2} (i)-(iii) and let $(u_0,v_0)^\top=(0,0)^\top$. Then, for $1\leq i,j\leq p$, $1\leq k,l\leq q,$ we observe
    \begin{align}
        \sup_{x\in\mathcal{J}}\operatorname{Var}\left(\int_0^Tu_{\delta,x}^{\Delta_i} (t)u_{\delta,x}^{\Delta_j}(t)\diff t\right)&=o(\delta^{-4(\alpha_i+\alpha_j-\alpha_1-\beta_1)});\label{eq: var uDelta}\\
        \sup_{x\in\mathcal{J}}\operatorname{Var}\left(\int_0^Tv_{\delta,x}^{\Delta_k}(t)v_{\delta,x}^{\Delta_l}(t)\diff t\right)&=o(\delta^{-4(\beta_k+\beta_l-\beta_1)});\label{eq: var v}\\
        \sup_{x\in\mathcal{J}}\operatorname{Var}\left(\int_0^Tu_{\delta,x}^{\Delta_i} (t)v_{\delta,x}^{\Delta_k}(t)\diff t\right)&=o(\delta^{-4(\alpha_i+\beta_k-\beta_1-\alpha_1/2)}).\label{eq: var cov}
    \end{align}
\end{lemma}
\begin{proof}[Proof of \Cref{lem: convVar}]
    We only show the assertion for \eqref{eq: var uDelta}. The other two statements \eqref{eq: var v} and \eqref{eq: var cov} follow in the same way. By Wick's formula \cite[Theorem 1.28]{janson_gaussian_1997} it holds
    \begin{align}
    \label{eq: variance_representation}
        &\delta^{4(\alpha_i+\alpha_j-\alpha_1-\beta_1)}\operatorname{Var}\left(\int_0^Tu_{\delta,x}^{\Delta_i} (t)u_{\delta,x}^{\Delta_j}(t)\diff t\right)=\delta^{4(\alpha_1+\alpha_j-\alpha_1-\beta_1)}(V_1+V_2)
    \end{align}
    with
    \begin{align*}
        V_1&\coloneqq V((-\Delta)^{\alpha_i}K_{\delta,x},(-\Delta)^{\alpha_i}K_{\delta,x},(-\Delta)^{\alpha_j}K_{\delta,x},(-\Delta)^{\alpha_j}K_{\delta,x}),\\
        V_2&\coloneqq V((-\Delta)^{\alpha_i}K_{\delta,x},(-\Delta)^{\alpha_j}K_{\delta,x},(-\Delta)^{\alpha_j}K_{\delta,x},(-\Delta)^{\alpha_i}K_{\delta,x}),
    \end{align*}
    and
    $$V(w,w',z,z')\coloneqq\int_0^T\int_0^T\operatorname{Cov}(\sca{u(t)}{w},\sca{u(s)}{w'})\operatorname{Cov}(\sca{u(t)}{z},\sca{u(s)}{z'})\diff s\diff t,$$
    for $w,w',z,z'\in L^2(\Lambda)$. 
    By \Cref{lemma:covariance_structure} and rescaling, we obtain the representation
    \begin{align}
        &\delta^{4(\alpha_i+\alpha_j-\alpha_1-\beta_1)}V_1 \nonumber\\
        &= \int_{0}^{T}\int_{0}^{T}\mathrm{Cov}(u_{\delta, x}^{\Delta_i}(t),u_{\delta,x}^{\Delta_i}(s))\mathrm{Cov}(u_{\delta,x}^{\Delta_j}(t),u_{\delta,x}^{\Delta_j}(s)) \mathrm{d}s \mathrm{d}t \label{eq: pure_cov_rep}\\
    &=2\delta^{2\beta_1}\int_0^T\int_0^{t\delta^{-2\beta_1}}\left(\int_0^{t\delta^{-2\beta_1}-s}f_{i,i}(s,s')\diff s'\right) \left(\int_0^{t\delta^{-2\beta_1}-s}f_{j,j}(s,s'')\diff s''\right)\diff s\diff t, \label{eq: f_ij representation}
    \end{align}
    where, for $s,s'\in [0, T\delta^{-2\beta_1}]$, we have set
    \begin{align*}
        f_{i,j}(s,s')&\coloneqq\langle{e^{(s+s')B_{\eta,\delta,x}/2}\sin(\delta^{-\alpha_1+2\beta_1}(s+s')L_{\theta,\eta,\delta,x}^{1/2})L_{\theta,\eta,\delta,x}^{-1/2}(-\Delta_{\delta,x})^{\alpha_i}K},\\
        &\hspace{4mm}{e^{s'B_{\eta,\delta,x}/2}\sin(\delta^{-\alpha_1+2\beta_1}(s')L_{\theta,\eta,\delta,x}^{1/2})L_{\theta,\eta,\delta,x}^{-1/2}(-\Delta_{\delta,x})^{\alpha_j}K}\rangle_{L^2(\Lambda_{\delta,x})}.
    \end{align*}
    In case that $\beta_1=0$, we use the pointwise convergences $f_{i,i}(s,s')\rightarrow 0$ and $f_{j,j}(s,s'')\rightarrow0,$ 
    given by the slow-fast orthogonality in \Cref{lem: scaling limit}(ii), and dominated convergence over fixed finite time intervals to prove the claim directly from the representation \eqref{eq: pure_cov_rep}. If, however, $\beta_1>0$, we use \Cref{ass: mainass2} (ii), i.e. $K=\Delta_0^{\lceil\alpha_1\rceil}\tilde K$, and \Cref{lem: semigroupbound} such that 
    \begin{align*}
        \sup_{x\in\mathcal{J}}|f_{i,i}(s',s)|&\lesssim (1\wedge (s+s')^{-(\alpha_i+\alpha_1/2)/\beta_1})(1\wedge s^{-(\alpha_i+\alpha_1/2)/\beta_1})\\
        &\lesssim(1\wedge s'^{-1})(1\wedge s^{-1}).
    \end{align*}
    Thus implies $\sup_{x\in\mathcal{J}}|V_1|=O(\delta^{-4(\alpha_i+\alpha_j-\alpha_1-\beta_1)}\delta^{2\beta_1}\log(\delta^{-2\beta_1}))=o(\delta^{-4(\alpha_i+\alpha_j-\alpha_1-\beta_1)})$. The arguments for $V_2$ follow in the same way be replacing $f_{i,i}$ and $f_{j,j}$ with $f_{i,j}$ and $f_{j,i}$ in \eqref{eq: f_ij representation}, respectively.  The assertion follows in view of \eqref{eq: variance_representation}.
\end{proof} 
\begin{lemma}[{Bounds on the initial condition}]
\label{lem: initial}
    Grant \Cref{ass: mainass2} (i) (ii) and (iv). Then, for $1\leq i\leq p$, $1\leq j\leq q$, we have
    \begin{itemize}
        \item [(i)] $\sup_{x\in\mathcal{J}}\delta^{4\alpha_i-2\alpha_1-2\beta_1}\left(\int_0^T\sca{M(t)u_0+N(t)v_0}{(-\Delta)^{\alpha_i} K_{\delta,x}}^2\diff t\right)=o(1);$
        \item [(ii)] $\sup_{x\in\mathcal{J}}\delta^{4\beta_j-2\beta_1}\left(\int_0^T\sca{A_\theta N(t)u_0+(M(t)+B_\eta(N(t))v_0}{(-\Delta)^{\beta_j}K_{\delta,x}}^2\diff t\right)=o(1).$
    \end{itemize}
\end{lemma}
\begin{proof}[Proof of \Cref{lem: initial}]
    \begin{itemize}
        \item []
        \item [(i)] Define the reverse scaling operation for $z\in L^2(\R^d)$ via
        $$z_{(\delta,x)^{-1}}(y)\coloneqq\delta^{d/2}z(x+\delta y),\quad y\in\R^d.$$
        The rescaling \Cref{lem: rescaling}, self-adjointness and the commutativity of operators imply
        \begin{align*}
            &\sca{M(t)u_0}{(-\Delta)^{\alpha_i} K_{\delta,x}}^2\\&\qquad=\delta^{-4\alpha_i}\sca{(u_0)_{(\delta,x)^{-1}}}{M_{\delta,x}(t)(-\Delta_{\delta,x})^{\alpha_i} K}_{L^2(\Lambda_{\delta,x})}^2\\
            &\qquad=\delta^{-4\alpha_i+4\alpha_1}\sca{((-\Delta)^{\alpha_1} u_0)_{(\delta,x)^{-1}}}{M_{\delta,x}(t) (-\Delta_{\delta,x})^{\alpha_i-\alpha_1}K}_{L^2(\Lambda_{\delta,x})}^2\\
            &\qquad\lesssim \delta^{-4\alpha_i+4\alpha_1}\norm{(-\Delta)^{\alpha_1} u_0}^2_{L^2(\Lambda)}\norm{e^{t\delta^{-2\beta_1}B_{\eta,\delta,x}} (-\Delta_{\delta,x})^{\alpha_i-\alpha_1}K}^2_{L^2(\Lambda_{\delta,x})}\\
            &\qquad\lesssim\delta^{-4\alpha_i+4\alpha_1}\norm{e^{t\delta^{-2\beta_1}B_{\eta,\delta,x}} (-\Delta_{\delta,x})^{\alpha_i-\alpha_1}K}^2_{L^2(\Lambda_{\delta,x})}.
        \end{align*}
        Thus, using \Cref{lem: semigroupbound} and that $K=\Delta_0^{\lceil \alpha_1\rceil}\tilde K,$ we obtain the upper bound
        \begin{align*}
            &\sup_{x\in\mathcal{J}}\delta^{4\alpha_i-2\alpha_1-2\beta_1}\int_0^T\sca{M(t)u_0}{(-\Delta)^{\alpha_i} K_{\delta,x}}^2\diff t\\
            &\lesssim \delta^{2\alpha_1}\int_0^{T\delta^{-2\beta_1}}\sup_{x\in\mathcal{J}}\norm{e^{tB_{\eta,\delta,x}} (-\Delta_{\delta,x})^{\alpha_i-\alpha_1}K}^2_{L^2(\Lambda_{\delta,x})}\diff t\\
            &\lesssim\delta^{2\alpha_1}\int_0^{T\delta^{-2\beta_1}}1\wedge t^{-\alpha_i/\beta_1}\diff t\\
            &=O(\delta^{2(\alpha_1-\beta_1)})=o(1).
        \end{align*}
        Similarly,
        \begin{align*}
            &\sca{N(t)v_0}{(-\Delta)^{\alpha_i} K_{\delta,x}}^2\\&\qquad=\delta^{-4\alpha_i}\sca{(v_0)_{(\delta,x)^{-1}}}{N_{\delta,x}(t)(-\Delta_{\delta,x})^{\alpha_i} K}_{L^2(\Lambda_{\delta,x})}^2\\
            &\qquad=\delta^{-4\alpha_i+2\alpha_1}\sca{((-\Delta)^{\alpha_1/2}v_0)_{(\delta,x)^{-1}}}{N_{\delta,x}(t)(-\Delta_{\delta,x})^{\alpha_i-\alpha_1/2}\Delta K}_{L^2(\Lambda_{\delta,x})}^2\\
            &\qquad\lesssim \delta^{-4\alpha_i+4\alpha_1}\norm{e^{t\delta^{-2\beta_1}B_{\eta,\delta,x}}L_{\theta,\eta,\delta,x}^{-1/2}(-\Delta_{\delta,x})^{\alpha_i-\alpha_1/2} K}^2_{L^2(\Lambda_{\delta,x})}
        \end{align*}
        Hence, 
        \begin{align*}
            &\sup_{x\in\mathcal{J}}\delta^{4\alpha_i-2\alpha_1-2\beta_1}\int_0^T\sca{N(t)v_0}{\Delta K_{\delta,x}}^2\diff t\\
            &\lesssim\delta^{2\alpha_1}\int_0^{T\delta^{-2\beta_1}}\sup_{x\in\mathcal{J}}\norm{e^{tB_{\eta,\delta,x}}L_{\theta,\eta,\delta,x}^{-1/2}(-\Delta_{\delta,x})^{\alpha_i-\alpha_1/2} K}^2_{L^2(\Lambda_{\delta,x})}\diff t\\
            &=O(\delta^{2(\alpha_1-\beta_1)})=o(1),
        \end{align*}
        proving the assertion.
        \item [(ii)] The steps from (i) can be repeated, resulting in 
        \begin{align*}
            &\sup_{x\in\mathcal{J}}\delta^{4\beta_j-2\beta_1}\left(\int_0^T\sca{A_\theta N(t)u_0+(M(t)+B_\eta(N(t))v_0}{(-\Delta)^{\beta_j}K_{\delta,x}}^2\diff t\right)\\
            &\qquad=O(\delta^{2(\alpha_1-\beta_1)})=o(1).\qedhere
        \end{align*}
    \end{itemize}
\end{proof}
\subsection{Proof of the CLT}
\begin{proof}[Proof of \Cref{result: convergence_joint_estimator2}]
    \begin{enumerate}[label=(\roman*)]
    \item[] 
    \item Assume first that $(u_0,v_0)^\top=(0,0)^\top.$ For any $1\leq i,j\leq p+q$, we obtain from \Cref{lem: convFisher} and \Cref{lem: convVar} that 
    \begin{align*}
        (\rho_\delta\mathcal{I}_\delta\rho_\delta)_{ij}&=\rho_{ii}\rho_{jj}\sum_{k=1}^N\int_0^T(Y_{ \delta, k}(t))_i(Y_{\delta,k}(t))_j\diff t\\
        &=(\Sigma_{\theta,\eta})_{ij}+o_{\PP}(1).
    \end{align*}
    This yields for zero initial conditions the convergence 
    \begin{equation}
        \label{eq: conv FisherProb}
        (\rho_\delta{\mathcal{I}}_\delta\rho_\delta)\stackrel{\PP}{\rightarrow}\Sigma_\theta,\quad\delta\rightarrow0.
    \end{equation}
    In order to extend this result to a general initial condition $(u_0,v_0)^\top$ satisfying \Cref{ass: mainass2} (iii), let $(\bar{u}(t),\bar{v}(t))^\top$ be defined as $(u(t),v(t))^\top,$ but starting in $(0,0)^\top$ such that for $z\in L^2(\Lambda)$
    \begin{align*}
        \sca{u(t)}{z}&=\sca{\bar{u}(t)}{z}+\sca{M(t)u_0}{z}+\sca{N(t)v_0}{z},\\
        \sca{v(t)}{z}&=\sca{\bar{v}(t)}{z}+\sca{A_\theta N(t)u_0}{z}+\sca{(M(t)+B_\eta N(t))v_0}{z}.
    \end{align*} 
    If $\bar{\mathcal{I}}_\delta$ corresponds to the observed Fisher information with zero initial condition, then by the Cauchy-Schwarz inequality
    $$\left|(\rho_\delta\mathcal{I}_\delta\rho_\delta)_{ij}-(\rho_\delta\bar{\mathcal{I}}_\delta\rho_\delta)_{ij}\right|\lesssim(\rho_\delta\bar{\mathcal{I}}_\delta\rho_\delta)_{ii}^{1/2}w_j^{1/2}+(\rho_\delta\bar{\mathcal{I}}_\delta\rho_\delta)_{jj}^{1/2}w_i^{1/2}+w_i^{1/2}w_j^{1/2},$$
    for all $1\leq i,j\leq p+q,$ where 
    \begin{equation*}
        w_i=\begin{cases}
            \sup_{x\in\mathcal{J}}N\rho_{ii}^2\left(\int_0^T\sca{M(t)u_0+N(t)v_0}{(-\Delta)^{\alpha_i} K_{\delta,x}}^2\diff t\right),\quad 1\leq i\leq p,\\
            \sup_{x\in\mathcal{J}}N\rho_{ii}^2\left(\int_0^T\sca{A_\theta N(t)u_0+(M(t)+B_\eta(N(t))v_0}{(-\Delta)^{\beta_i} K_{\delta,x}}^2\diff t\right),\quad \text{else}.
        \end{cases}
    \end{equation*}
    By the first part, $(\rho_\delta\bar{\mathcal{I}}_\delta\rho_\delta)_{ii}$ is bounded in probability and \Cref{lem: initial} shows $w_i=o(1)$. Hence, we obtain \eqref{eq: conv FisherProb} also in the case of non-zero initial conditions.
    
    Due to \Cref{ass: mainass2} (i)-(iii), $\Sigma_{\theta,\eta}$ is well-defined as all entries are finite. Regarding invertibility, note first that $\Sigma_{\theta,\eta}$ is invertible if and only if both $\Sigma_{1,\theta,\eta}$ and $\Sigma_{2,\theta,\eta}$ are invertible. We only argue that $\Sigma_{1,\theta,\eta}$ is invertible as the argument for $\Sigma_{2, \theta,\eta}$ is identical.
    Let $\lambda\in\R^p$ such that 
    \begin{align*}
        0=\sum_{i,j=1}^p\lambda_i\lambda_j(\Sigma_{1,\theta,\eta})_{ij}\iff 0=\sum_{i,j=1}^p\lambda_i\lambda_j\norm{(-\Delta)^{(\alpha_i+\alpha_j-\alpha_1-\beta_1)}K}^2_{L^2(\R^d)}.
    \end{align*}
    Now,
    \begin{align*}
        0&=\sum_{i,j=1}^p\lambda_i\lambda_j\norm{(-\Delta)^{(\alpha_i+\alpha_j-\alpha_1-\beta_1)}K}^2_{L^2(\R^d)}\\
        &=\sca{\sum_{i=1}^p\lambda_i(-\Delta)^{\alpha_i-(\alpha_1+\beta_1)/2}K}{\sum_{i=1}^p\lambda_i(-\Delta)^{\alpha_i-(\alpha_1+\beta_1)/2}K}_{L^2(\R^d)},
    \end{align*}
    hence $\sum_{i=1}^p\lambda_i(-\Delta)^{\alpha_i-(\alpha_1+\beta_1)/2}K=0.$ Since  the functions $(-\Delta)^{\alpha_i-(\alpha_1+\beta_1)/2}K$ ,$1\leq i\leq p,$ are linearly independent by \Cref{ass: mainass2} (iii), $\Sigma_{1,\theta,\eta}$ is invertible.
        \item We refer to \cite[Theorem 2.3]{altmeyer_anisotrop2021} for a detailed proof of the CLT in the case of the perturbed stochastic heat equation, which relies on a general multivariate martingale central limit theorem. All steps translate directly into our setting due to the stochastic convergence $\rho_\delta\mathcal{I}_\delta\rho_\delta\stackrel{\PP}{\rightarrow}\Sigma_{\theta,\eta}$ from (i).
    \end{enumerate}
\end{proof}
\paragraph*{Acknowledgement}
AT gratefully acknowledges the financial support of Carlsberg Foundation Young Researcher Fellowship grant CF20-0604. The research of EZ has been partially funded by Deutsche Forschungsgemeinschaft (DFG)—SFB1294/1-318763901.
\newpage
\bibliography{references}

\begin{thebibliography}{}

\bibitem[Aihara \& Bagchi, 1991]{aihara_parameter_1991}
Aihara, S., I. \& Bagchi, A. (1991).
\newblock Parameter identification for hyperbolic stochastic systems.
\newblock {\em Journal of mathematical analysis and applications}, 160(2),
  485--499.

\bibitem[Aihara, 1994]{aihara_identification_1994}
Aihara, S.~I. (1994).
\newblock Identification of discontinuous parameter in stochastic hyperbolic
  systems.
\newblock {\em IFAC Proceedings Volumes}, 27(8), 191--196.

\bibitem[Altmeyer et~al., 2022]{altmeyer_parameter_2020}
Altmeyer, R., Bretschneider, T., Janák, J., \& Reiß, M. (2022).
\newblock Parameter {Estimation} in an {SPDE} {Model} for {Cell}
  {Repolarisation}.
\newblock {\em SIAM/ASA Journal on Uncertainty Quantification}, 10(1),
  179--199.

\bibitem[Altmeyer et~al., 2023]{altmeyer_parameterSemi_2020}
Altmeyer, R., Cialenco, I., \& Pasemann, G. (2023).
\newblock Parameter estimation for semilinear {SPDEs} from local measurements.
\newblock {\em Bernoulli}, 29(3), 2035--2061.

\bibitem[Altmeyer \& Reiß, 2021]{altmeyer_nonparametric_2020}
Altmeyer, R. \& Reiß, M. (2021).
\newblock Nonparametric estimation for linear {SPDEs} from local measurements.
\newblock {\em Annals of Applied Probability}, 31(1), 1--38.

\bibitem[Altmeyer et~al., 2024]{altmeyer_anisotrop2021}
Altmeyer, R., Tiepner, A., \& Wahl, M. (2024).
\newblock {Optimal parameter estimation for linear SPDEs from multiple
  measurements}.
\newblock {\em The Annals of Statistics}, 52(4), 1307 -- 1333.

\bibitem[Arendt et~al., 2001]{arendtVectorvaluedLaplaceTransforms2001}
Arendt, W., Batty, C. J.~K., Hieber, M., \& Neubrander, F. (2001).
\newblock {\em Vector-Valued {{Laplace Transforms}} and {{Cauchy Problems}}}.
\newblock {Springer}.

\bibitem[Aspelmeier et~al., 2015]{aspelmeier_modern_2015}
Aspelmeier, T., Egner, A., \& Munk, A. (2015).
\newblock Modern statistical challenges in high-resolution fluorescence
  microscopy.
\newblock {\em Annual Reviews of Statistics and Its Applications}, 2, 163--202.

\bibitem[Backer \& Moerner, 2014]{backer_extending_2014}
Backer, A.~S. \& Moerner, W.~E. (2014).
\newblock Extending {Single}-{Molecule} {Microscopy} {Using} {Optical}
  {Fourier} {Processing}.
\newblock {\em The Journal of Physical Chemistry B}, 118(28), 8313--8329.

\bibitem[Chen \& Russell, 1982]{russell_damping_1982}
Chen, G. \& Russell, D.~L. (1982).
\newblock A mathematical model for linear elastic systems with structural
  damping.
\newblock {\em Quarterly of Applied Mathematics}, 39(4), 433--454.

\bibitem[Chen \& Triggiani, 1989]{chen_damping_1989}
Chen, S.~P. \& Triggiani, R. (1989).
\newblock {Proof of extensions of two conjectures on structural damping for
  elastic systems.}
\newblock {\em Pacific Journal of Mathematics}, 136(1), 15 -- 55.

\bibitem[Chong, 2020]{chong_high-frequency_2020}
Chong, C. (2020).
\newblock High-frequency analysis of parabolic stochastic {PDEs}.
\newblock {\em Annals of Statistics}, 48(2), 1143--1167.

\bibitem[Da~Prato \& Zabczyk, 2014]{da_prato_stochastic_2014}
Da~Prato, G. \& Zabczyk, J. (2014).
\newblock {\em Stochastic equations in infinite dimensions}.
\newblock Cambridge University Press.

\bibitem[Favini \& Obrecht, 1991]{Favini_conditions_1991}
Favini, A. \& Obrecht, E. (1991).
\newblock {Conditions for parabolicity of second order abstract differential
  equations}.
\newblock {\em Differential and Integral Equations}, 4(5), 1005 -- 1022.

\bibitem[Gaudlitz \& Rei{\ss}, 2023]{gaudlitz_estimation_2023}
Gaudlitz, S. \& Rei{\ss}, M. (2023).
\newblock {Estimation for the reaction term in semi-linear SPDEs under small
  diffusivity}.
\newblock {\em Bernoulli}, 29(4), 3033 -- 3058.

\bibitem[Hildebrandt \& Trabs, 2021]{hildebrandt_parameter_2019}
Hildebrandt, F. \& Trabs, M. (2021).
\newblock Parameter estimation for {SPDEs} based on discrete observations in
  time and space.
\newblock {\em Electronic Journal of Statistics}, 15, 2716--2776.

\bibitem[Huebner \& Rozovskii, 1995]{huebner_asymptotic_1995}
Huebner, M. \& Rozovskii, B. (1995).
\newblock On asymptotic properties of maximum likelihood estimators for
  parabolic stochastic {PDE}'s.
\newblock {\em Probability Theory and Related Fields}, 103(2), 143--163.

\bibitem[Janson, 1997]{janson_gaussian_1997}
Janson, S. (1997).
\newblock {\em Gaussian {Hilbert} {Spaces}}.
\newblock Cambridge University Press.

\bibitem[Janák \& Reiß, 2024]{janak_2023_multiplicative}
Janák, J. \& Reiß, M. (2024).
\newblock Parameter estimation for the stochastic heat equation with
  multiplicative noise from local measurements.
\newblock {\em Stochastic Processes and their Applications}, 175, 104385.

\bibitem[Kov\'{a}cs et~al., 2010]{kovacs_finite_2010}
Kov\'{a}cs, M., Larsson, S., \& Saedpanah, F. (2010).
\newblock Finite element approximation of the linear stochastic wave equation
  with additive noise.
\newblock {\em SIAM Journal on Numerical Analysis}, 48(2), 408--427.

\bibitem[K{\"u}chler \& S{\o}rensen, 1997]{küchler_exponential_1997}
K{\"u}chler, U. \& S{\o}rensen, M. (1997).
\newblock {\em Exponential families of stochastic processes}.
\newblock Springer.

\bibitem[Kutoyants, 2013]{kutoyants_statistical_2013}
Kutoyants, Y.~A. (2013).
\newblock {\em Statistical {Inference} for {Ergodic} {Diffusion} {Processes}}.
\newblock Springer.

\bibitem[Leissa, 1969]{leissa_plate_1969}
Leissa, A., W. (1969).
\newblock {\em {Vibration of Plates}}.
\newblock National Aeronautics and Space Administration.

\bibitem[Liu \& Lototsky, 2010]{lototsky_multichannel_2010}
Liu, W. \& Lototsky, S. (2010).
\newblock Parameter estimation in hyperbolic multichannel models.
\newblock {\em Asymptotic Analysis}, 68, 223--248.

\bibitem[Lord et~al., 2014]{lord_introduction_2014}
Lord, G.~J., Powell, C.~E., \& Shardlow, T. (2014).
\newblock {\em An {Introduction} to {Computational} {Stochastic} {PDEs}}.
\newblock Cambridge University Press.

\bibitem[Lototsky, 2009]{lototsky_statistical_2009}
Lototsky, S.~V. (2009).
\newblock Statistical inference for stochastic parabolic equations: a spectral
  approach.
\newblock {\em Publicacions Matemàtiques}, 53(1), 3--45.
\newblock Publisher: Universitat Autònoma de Barcelona, Departament de
  Matemàtiques.

\bibitem[Maslowski \& Posp{\'\i}{\v{s}}il, 2008]{maslowski_ergodicity_2007}
Maslowski, B. \& Posp{\'\i}{\v{s}}il, J. (2008).
\newblock Ergodicity and parameter estimates for infinite-dimensional
  fractional ornstein-uhlenbeck process.
\newblock {\em Applied Mathematics and Optimization}, 57, 401--429.

\bibitem[Melnikova \& Filinkov, 2001]{melnikova_abstract_2001}
Melnikova, Irina, V. \& Filinkov, A. (2001).
\newblock {\em Abstract Cauchy Problems: Three Approaches}.
\newblock Chapman and Hall.

\bibitem[Pandolfi, 2021]{pandolfiSystemsPersistentMemory2021a}
Pandolfi, L. (2021).
\newblock {\em Systems with {{Persistent Memory}}: {{Controllability}},
  {{Stability}}, {{Identification}}}, volume~54 of {\em Interdisciplinary
  {{Applied Mathematics}}}.
\newblock {Springer International Publishing}.

\bibitem[Pasemann \& Stannat, 2020]{pasemann_drift_2020}
Pasemann, G. \& Stannat, W. (2020).
\newblock Drift estimation for stochastic reaction-diffusion systems.
\newblock {\em Electronic Journal of Statistics}, 14(1), 547--579.

\bibitem[Reddy, 2006]{reddy_plate_2006}
Reddy, J., N. (2006).
\newblock {\em {Theory and Analysis of Elastic Plates and Shells (2nd. ed.)}}.
\newblock CRC Press.

\bibitem[Reiß et~al., 2023]{reiß_2023_change}
Reiß, M., Strauch, C., \& Trottner, L. (2023).
\newblock Change point estimation for a stochastic heat equation.
\newblock arXiv:2307.10960.

\bibitem[Rudin, 1991]{rudin_functional_1991}
Rudin, W. (1991).
\newblock {\em Functional analysis}.
\newblock McGraw-Hill.

\bibitem[Schm{\"u}dgen, 2012]{schmudgenUnboundedSelfadjointOperators2012}
Schm{\"u}dgen, K. (2012).
\newblock {\em Unbounded {{Self-adjoint Operators}} on {{Hilbert Space}}},
  volume 265 of {\em Graduate {{Texts}} in {{Mathematics}}}.
\newblock {Springer Netherlands}.

\bibitem[Shimakura, 1992]{shimakura_partial_1992}
Shimakura, N. (1992).
\newblock {\em Partial {Differential} {Operators} of {Elliptic} {Type}}.
\newblock American Mathematical Soc.

\bibitem[Sova, 1966]{sova1966cosine}
Sova, M. (1966).
\newblock {\em Cosine operator functions}.
\newblock Instytut Matematyczny Polskiej Akademi Nauk (Warszawa).

\bibitem[Strauch \& Tiepner, 2024]{strauch_velocity_2023}
Strauch, C. \& Tiepner, A. (2024).
\newblock Nonparametric velocity estimation in stochastic convection-diffusion
  equations from multiple local measurements.
\newblock arXiv:2402.08353.

\bibitem[Tiepner \& Trottner, 2025]{tiepner_change_24}
Tiepner, A. \& Trottner, L. (2025).
\newblock Multivariate change estimation for a stochastic heat equation from
  local measurements.
\newblock arXiv:2409.15059.

\bibitem[Tonaki et~al., 2023]{tonaki2022parameter}
Tonaki, Y., Kaino, Y., \& Uchida, M. (2023).
\newblock Parameter estimation for linear parabolic spdes in two space
  dimensions based on high frequency data.
\newblock {\em Scandinavian Journal of Statistics}, 50(4), 1568--1589.

\bibitem[Weidmann, 1997]{weidmannStrongOperatorConvergence1997}
Weidmann, J. (1997).
\newblock Strong operator convergence and spectral theory of ordinary
  differential operators.
\newblock {\em ZESZYTY NAUKOWE-UNIWERSYTETU JAGIELLONSKIEGO-ALL SERIES}, 1208,
  153--163.

\bibitem[Ziebell, 2024]{ziebell_wave23}
Ziebell, E. (2024).
\newblock Non-parametric estimation for the stochastic wave equation.
\newblock arXiv:2404.18823.

\end{thebibliography}
\bibliographystyle{apalike2}
\end{document}